\documentclass[amsmath,amstex,secnumarabic,floatfix,amssymb,aps,nofootinbib,nobibnotes,letterpaper,11pt,tightenlines]{revtex4}

\usepackage{times}

\usepackage{geometry}
\usepackage{amssymb}
\usepackage{latexsym, amsmath, amscd,amsthm}
\usepackage{graphicx}
\usepackage[percent]{overpic}
\usepackage{pdfsync}
\usepackage{units}
\usepackage{hyperref}
\usepackage{diagrams}

\newtheorem{theorem}{Theorem}

\newtheorem{lemma}[theorem]{Lemma}
\newtheorem{proposition}[theorem]{Proposition}
\newtheorem{corollary}[theorem]{Corollary}

\theoremstyle{definition}
\newtheorem{definition}[theorem]{Definition}

\newtheorem{remark}[theorem]{Remark}

\def\defn#1{Definition~\ref{def:#1}}
\def\thm#1{Theorem~\ref{thm:#1}}
\def\lem#1{Lemma~\ref{lem:#1}}

\def\rmark#1{Remark~\ref{rmark:#1}}

\newcommand{\R}{\mathbb{R}}
\newcommand{\Z}{\mathbb{Z}}

\newcommand{\simpR}{\operatorname{Simp}_R}
\newcommand{\ds}{\displaystyle}

\newcommand{\TC}{\operatorname{TC}}
\newcommand{\arc}[1]{\gamma_{#1}}

\newcommand{\Len}{\operatorname{Len}}
\newcommand{\bdry}{\partial}

\newcommand{\transverse}{\pitchfork}
\newcommand{\SO}{\operatorname{SO}}

\newcommand{\Span}{\operatorname{Span}}

\newcommand{\pij}{\pi_{ij}}
\newcommand{\sijk}{s_{ijk}}
\newcommand\norm[1]{\left| #1 \right|}
\newcommand{\cnm}{C_n[M]}
\newcommand{\cnmo}{C_n(M)}

\newcommand{\cn}{\protect\overrightarrow{\mathbf{n}} }
\newcommand{\p}{\protect\overrightarrow{\mathbf{p}} }
\newcommand{\q}{\protect\overrightarrow{\mathbf{q}} }
\newcommand{\cu}{\protect\overrightarrow{\mathbf{u}} }
\newcommand{\cv}{\protect\overrightarrow{\mathbf{v}} }
\newcommand{\cw}{\protect\overrightarrow{\mathbf{w}} }
\newcommand{\x}{\protect\overrightarrow{\mathbf{x}} }

\newcommand{\bfa}{\mathbf{a}}

\newcommand{\bfe}{\mathbf{e}}
\newcommand{\bfp}{\mathbf{p}}
\newcommand{\bfq}{\mathbf{q}}
\newcommand{\bfr}{\mathbf{r}}

\newcommand{\bfu}{\mathbf{u}}
\newcommand{\bfv}{\mathbf{v}}
\newcommand{\bfw}{\mathbf{w}}

\newcommand{\bfx}{\mathbf{x}}

\newcommand{\cnr}{C_n[\R^k]}
\newcommand{\cnro}{C_n(\R^k)}
\newcommand{\cfrt}{C_4[\R^2]}

\newcommand{\cfr}{C_4[\R^k]}
\newcommand{\cfro}{C_4(\R^k)}
\newcommand{\qcfr}{\hat{C}_4[\R^k]}

\newcommand{\cfs}{C_4[S^1]}

\newcommand{\cns}{C_n[S^1]}
\newcommand{\cnso}{C_n(S^1)}
\newcommand{\cfg}{C_4[\gamma]}

\newcommand{\cfog}{C^0_4[\gamma]}

\newcommand{\qcfog}{\hat{C}^0_4[\gamma]}

\newcommand{\cng}{C_n[\gamma]}

\newcommand{\ccng}{C^0_n[\gamma]}

\newcommand{\cnf}{C_n[f]}

\newcommand{\slq}{Slq}

\newcommand{\qslq}{\widehat{S}lq}

\newcommand{\FTC}{\operatorname{FTC}}
\newcommand{\FTCWC}{\operatorname{FTCWC}}

\newcommand{\bdy}{\partial}
\newcommand{\cross}{\times}

\def\co{\colon\!}

\let\mgp=\marginpar \marginparwidth18mm \marginparsep1mm
\def\marginpar#1{\mgp{\raggedright\tiny #1}}

\let\lbl=\label
\def\label#1{\lbl{#1}\ifinner\else\marginpar{\ref{#1} #1}\ignorespaces\fi}

\newcommand{\sqrtext}{square-like quadrilateral }
\newcommand{\sqrtexts}{square-like quadrilaterals }
\newcommand{\sqrtextp}{square-like quadrilateral. }
\newcommand{\sqrtextsp}{square-like quadrilaterals. }

\newcommand{\paren}{parenthesization }
\newcommand{\parens}{parenthesizations }

\begin{document}
\title[$n$-gons and square-peg - preprint]
     {Transversality in Configuration Spaces and the ``Square-Peg'' theorem}
\author{Jason Cantarella}
\altaffiliation{University of Georgia, Mathematics Department, Athens GA}
\noaffiliation
\author{Elizabeth Denne}
\altaffiliation{Washington \& Lee University, Department of Mathematics, Lexington VA}
\noaffiliation
\author{John McCleary}
\altaffiliation{Vassar College, Mathematics Department, Poughkeepsie NY}
\noaffiliation

\begin{abstract} 
We prove a transversality ``lifting property'' for compactified configuration spaces as an application of the multijet transversality theorem: the submanifold of configurations of points on an arbitrary submanifold of Euclidean space may be made transverse to any submanifold of the configuration space of points in Euclidean space by an arbitrarily $C^1$-small variation of the initial submanifold, as long as the two submanifolds of compactified configuration space are boundary-disjoint. We use this setup to provide attractive proofs of the existence of a number of ``special inscribed configurations'' inside families of spheres embedded in $\R^n$ using differential topology. For instance, there is a $C^1$-dense family of smooth embedded circles in the plane where each simple closed curve has an odd number of inscribed squares, and there is a $C^1$-dense family of smooth embedded $(n-1)$-spheres in $\R^n$ where each sphere has a family of inscribed regular $n$-simplices with the homology of $O(n)$. 
\end{abstract}
\date{\today}
\maketitle

{\em Authors' note:} This paper will not be published in this form, but instead, has been split into three separate papers \cite{Slq,FTCWC,Simplices}. These papers will be published separately.  This paper has been cited extensively in the literature and so has been left on the arXiv as a reference to the reader.

\section{Introduction}\label{sect:intro}

Given a simple closed curve (a Jordan curve) $\gamma$ in $\R^2$, can we find four points on $\gamma$ that form a square? This question was posed by Toeplitz in 1911 \cite{Toeplitz:1911ta}
and it has drawn the attention of many mathematicians over the intervening century. Thinking of the Jordan curve as a ``round hole'', the problem has been affectionately dubbed the ``square-peg problem.'' We say that the square is \emph{inscribed} in $\gamma$ when the vertices lie on the curve. We do not require that the square lie entirely in the interior of the curve. Progress on the square-peg problem has chiefly been extension of the class of simple closed curves for which the square can be found. (The interested reader can find a number of survey articles such as \cite{MR1133201, Mat-SP-survey, Pak-Discrete-Poly-Geom}. )

\begin{figure}[t]
\hfill
\includegraphics[width=1.25in]{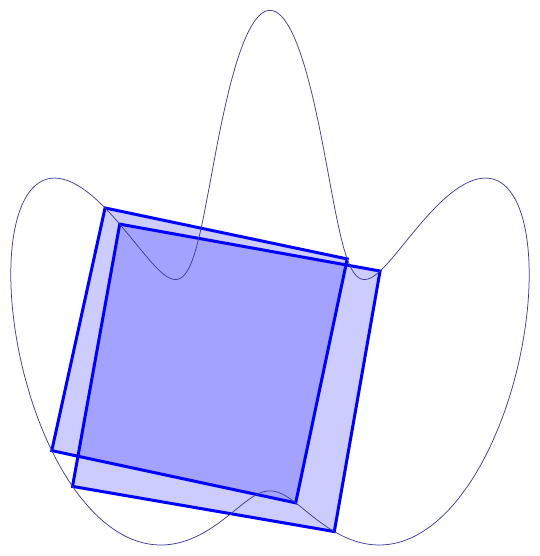}
\hfill
\includegraphics[width=1.25in]{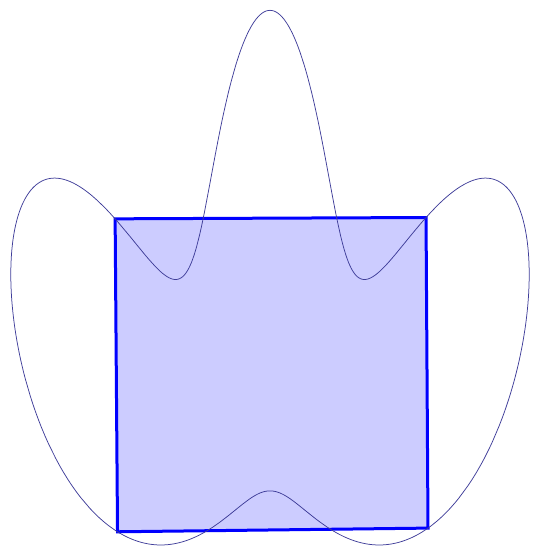}
\hfill
\includegraphics[width=1.25in]{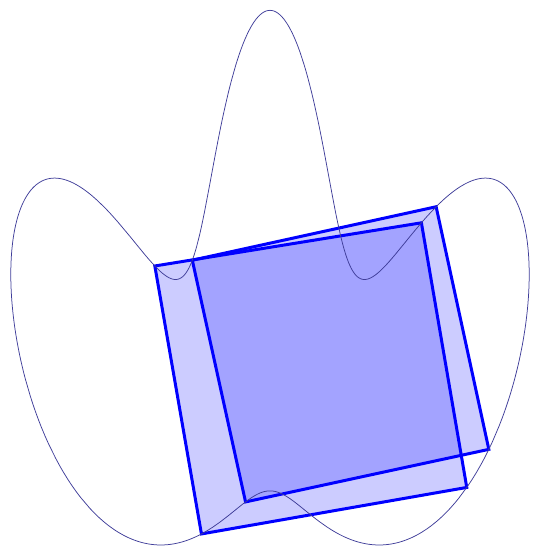}
\hfill
\hphantom{.}
\caption{This picture shows the five squares inscribed on an irregular planar curve. It turns out to be the case that the manifold of inscribed 4-tuples on this curve is transverse to the manifold of squares in the plane. Hence the squares are isolated and there are an odd number of squares. In general, our theorems guarantee only that a curve arbitrarily $C^1$-close to this one has this property.  
\label{fig:firstexample}}
\end{figure} 

Our goals are different. First, by placing the problem in the context of configuration spaces and their subspaces, we have opened up a set of tools from differential topology that allow fresh viewpoints through some powerful methods. Our conclusions include the previous work and show how differentiability assumptions can deliver a strong sense in which squares appear generically. The use of the multijet transversality theorem \cite{Golubitsky:1974iu} is new and holds promise for the application of differential topological methods to other configuration problems.

Here is the heart of our method: If we consider the (compactified) configuration space $C_4[\R^2]$ of 4-tuples of points in the plane as an 8-dimensional manifold with boundary (and corners), then Toeplitz's question can be rephrased more simply as a question about the intersections of the 4-dimensional submanifold of 4-tuples of points on $\gamma$, called $C_4[\gamma]$, with the 4-dimensional submanifold of squares in $\R^2$. 

We can see with a little effort that for a standard ellipse, these submanifolds intersect in four points corresponding to cyclic relabelings of a single inscribed square. It is therefore natural to try to show that squares are transverse to inscribed configurations in the ellipse and use an isotopy from the ellipse to $\gamma$ to connect the square on the ellipse to a cobordant family of squares on the target curve. 

This program requires us to face a few technical obstacles. First, the square might shrink away during the isotopy. We overcome this obstacle by analyzing the (compactified) boundary of our submanifolds of inscribed configurations and showing that, in a precise sense, $C^1$ curves do not admit infinitesimal squares. Second, we do not know that the submanifold of squares is transverse to the submanifold of inscribed configurations on $\gamma$. We may vary the submanifold of inscribed configurations using the standard transversality theorem for manifolds to make it transverse, of course, but there is no \emph{a priori} guarantee that the varied submanifold consists of inscribed configurations on any single curve. We deal with this problem by an application of the multijet transversality theorem~\cite{Golubitsky:1974iu}. Third, it turns out to be the case that the four intersections of the submanifold of squares with the submanifold of inscribed quadruples on the ellipse alternate sign. To count squares we must mod out by cyclic relabeling of vertices and pass to $\Z/2\Z$ intersection theory. 

The method we use for squares is an example of a general approach
to such ``special inscribed configuration'' problems: Show that the configurations one is looking for form a submanifold $Z$ of configuration space to establish smoothness, prevent ``shrink outs'' by showing that $Z$ is boundary-disjoint from the submanifold of inscribed configurations $C_n[\gamma]$, find the (transverse) intersection of $Z$ and $C_n[\gamma_0]$ explicitly in a base case, use our transversality theorem to conclude that a submanifold $\gamma'$ near the target submanifold $\gamma$ also has $C_n[\gamma'] \transverse Z$. Finally, use standard methods to build a isotopy from $C_n[\gamma_0]$ to $C_n[\gamma'] $ that is transverse to $Z$ at every step of the way.

In addition to counting squares (Theorem~\ref{thm:squarepeg}), we show as another example of these methods that there is a $k(k-1)/2$ dimensional family of inscribed simplices of any edgelength ratio in a generic embedding of $S^{k-1}$ in $\R^k$ (Theorem~\ref{thm:simplices}). 

It is important to note that while our results provide a unified and attractive view of this family of theorems about special inscribed configurations, they do not directly address the remaining open territory in Toeplitz's question: We give, in the Appendix, an extension of our results to prove that there exists at least one square on any embedded curve of finite total curvature without cusps, but this class of curves is certainly less general than the family of curves for which Stromquist proved the square peg theorem in \cite{MR1045781}.

\section{Configuration Spaces}
\label{sect:config}

The compactified configuration space of $n$ points in $\R^k$ is the natural setting for both the square-peg and inscribed polygon problem. A reader familiar with configuration spaces may skip much of this section. However we recommend paying attention to the notation we have used for the spaces, points in the spaces and the strata.  \defn{pijsijk}, \defn{config}, and \rmark{notation} are particularly useful. This section provides a brief overview of (compactified) configuration spaces. There are many versions of this classical material (see for instance~\cite{MR1259368,MR1258919}). We follow Sinha~\cite{newkey119} as this gives a geometric viewpoint appropriate to our setting.

\begin{definition}\label{def:openconfig}
Given an $m$-dimensional smooth manifold $M$, let  $M^{\times n}$ denote $n$ copies of $M$, and define $\cnmo$ to be the subspace of points  ${\mathbf{p}}=(p_1,\dots, p_n)\in M^{\times n}$ such that $p_j\neq p_k$ if $j\neq k$. Let $\iota$ denote the inclusion map of $\cnmo$ in~$M^{\times n}$. 
\end{definition}

The space $\cnmo$ is an open submanifold of $M^{\times n}$. Our goal is to compactify $\cnmo$ to a closed manifold with boundary and corners, which we will denote $\cnm$, without changing its homotopy type. The resulting manifold will be homeomorphic to $M^{\times n}$ with an open neighborhood of the fat diagonal
removed. Recall that the fat diagonal is the subset of $M^{\times n}$ of $n$-tuples for which (at least)
two entries are equal, that is, where
some collection of points comes together at a single point. The construction of $\cnm$ preserves information about the directions and relative rates of approach of each group of collapsing points. 

\newcommand{\otuple}[2]{\genfrac{[}{]}{0pt}{}{#1}{#2}}

\begin{definition}[\cite{newkey119} Definition 1.3]  \label{def:pijsijk}
Let $\otuple{n}{k}$ denote the number of ordered subsets of $k$ distinct elements of a set of size $n$. Given an ordered pair $(i,j)$ of $\{1,\dots,n\}$, let $\pij\co C_n(\R^m)\rightarrow S^{m-1}$ be the map that sends~$\bfp=(\bfp_1,\dots\bfp_n)$ to $\displaystyle\frac{\bfp_i-\bfp_j}{|\bfp_i-\bfp_j|}$, the unit vector in the direction of $\bfp_i-\bfp_j$.  Let $[0,\infty]$ be the one-point compactification of $[0,\infty)$.  Given an ordered triple $(i,j,k)$ of distinct elements in $\{1,\dots,n\}$, let $\sijk \co C_n(\R^m)\rightarrow [0,\infty]$ be the map which sends $\bfp$ to $\displaystyle\frac{|{\bfp_i}-{\bfp_j}|}{|\bfp_i-{\bfp_k}|}$. 
\end{definition}

To define configuration spaces for points in an arbitrary smooth ($C^\infty$) manifold $M$, we embed $M$ in~$\R^k$ so that $\cnmo$ is a subspace of $C_n(\R^k)$. We then compactify the space as follows:

\begin{definition} [\cite{newkey119} Definition 1.3] \label{def:config} Let $A_n[\R^k]$ be the product $(\R^k)^{n}\times (S^{k-1})^{\otuple{n}{2}} \times [0,\infty]^{\otuple{n}{3}}$.  Define $C_n[\R^k]$ to be the closure of the image of $C_n(\R^k)$ under the map
$$\alpha_n= \iota \times (\pij) \times (\sijk) \co C_n(\R^k)\rightarrow A_n[\R^k].$$
 If $M$ is smoothly embedded in $\R^k$, then $\cnmo$ is smoothly embedded in $C_n(\R^k)$ and we define $\cnm$ to be the closure of $\alpha_n(\cnmo)$ in $A_n[\R^k]$. In this case, we will refer to $A_n[\R^k]$ as $A_n[M]$ for convenience;  we denote the boundary of $\cnm$ by $\bdry\cnm=\cnm\setminus\cnmo$.
\end{definition}

We now summarize some of the important features of this construction, including the fact that $\cnm$ does not depend on the choice of embedding of $M$ in $\R^k$.
\begin{theorem}{\rm \label{thm:config}[cf.\cite{newkey119},~\cite{newkey118} Theorem 2.3]}
\begin{itemize}
\item $\cnm$ is a manifold with boundary and corners with interior $\cnmo$ having the same homotopy type as~$\cnm$. The topological type of $\cnm$ is independent of the embedding of $M$ in $\R^k$, and $\cnm$ is compact if $M$ is.
\item The inclusion of $\cnmo$ in $M^n$ extends to a surjective map  fron $\cnm$ to $M^n$ which is a homeomorphism over points in $\cnmo$. 
\end{itemize}
\end{theorem}

\begin{remark}\label{rmark:notation}
When discussing points in $\cnr$ or $\cnm$, it is easy to become confused. We pause to clarify notation.
\begin{itemize}
\item A point in $\R^k$ is denoted by $\bfx=(x_1, \dots, x_k)$, where each $x_i\in\R$.
\item Points in $(\R^k)^n$ are also denoted by $\bfx$, where $\bfx = (\bfx_1, \dots, \bfx_n)$ and each $\bfx_i\in\R^k$. (It will be clear from context which is meant.)
\item A point in $\cnr$ or $\cnm$, is denoted $\x$. 
\item At times, we will need to distinguish between the various entries of $\x\in\cnr$ or $\cnm$. 
In general,  
$$\x=(\bfx,(\pij)(\bfx), (\sijk)(\bfx)) =  (\bfx, \alpha(\bfx)),$$ where $\bfx=(\bfx_1,\dots,\bfx_n)\in (\R^k)^n$, and $ \alpha(\bfx) = ((\pij)(\bfx), (\sijk)(\bfx))$ gives the corresponding set of values in $(S^{k-1})^{\otuple{n}{2}}$ and $[0,\infty]^{\otuple{n}{3}}$.
 \end{itemize}
\end{remark}

The space $\cnm$ may be viewed as a polytope with a combinatorial structure based on the different ways groups of points in $M$ can come together. This structure defines a stratification of $\cnm$ into a collection of closed faces of various dimensions whose intersections are members of the collection. We will need to understand a bit of the structure of this collection, which is referred to as a \emph{stratification} of $\cnm$.
\begin{definition}[\cite{newkey118} Definition 2.4] A \emph{\paren} $\mathcal{P}$ of a set $T$ is an unordered collection $\{A_i\}$ of subsets of $T$ such that each subset contains at least 2 elements and two subsets are either disjoint or one is contained in the other. A \paren is denoted by a nested listing of the $A_i$ using parentheses. Let ${\bf Pa}(T)$ denote the set of \parens of $T$, and define an ordering on it by $\mathcal{P}\leq\mathcal{P}'$ if $\mathcal{P}\subseteq\mathcal{P}'$.
\end{definition}
For example, for $T=\{1,2,3,4\}$, $(12)(34)$ represents a \paren whose subsets are $\{1,2\}$ and $\{3,4\}$ while $((12)34)$ represents a \paren whose subsets are $\{1,2\}$ and $\{1,2,3,4\}$. 

We identify each parenthesization $\mathcal{P} = \{A_1, \dots, A_l\}$ of $\{1, \ldots, n\}$ with a closed subset $S_\mathcal{P}$ of $\bdy \cnm$ in our stratification of $\cnm$. The idea is that all the points in each $A_x$ collapse together, but if $A_x \subset A_y$, then the points in $A_x$ collapse ``faster'' than the points in $A_y$. Formally, this becomes the following condition: Let $\p = ((\bfp_1\dots,\bfp_n),(\pij)(\bfp), (\sijk)(\bfp))$ be a point in $A_n[M]$. Then $\p \in S_\mathcal{P}$ if 
\begin{itemize}
\item $\bfp_i = \bfp_j$ if and only if $i, j \in A_x$ for some $x$.
\item $s_{ijk} = 0$ (and hence $s_{ikj} = \infty$) if and only if $A_x \subset A_y$, $i, j \in A_x$ and $k \in A_y$. 
\end{itemize}
Sinha proves that a stratum $S_\mathcal{P}$ described by nested subsets $\{A_1,\dots, A_i\}$ has codimension~$i$ in $\cnm$. In the previous example $(12)$ has codimension~1, while $((12)34)$ and $(12)(34)$ have codimension~2. 

We notice that the definition of the $S_{\mathcal{P}}$ does not depend on the $\pi_{ij}$. In fact, for connected manifolds of dimension at least $2$, the combinatorial structure of the strata of $\cnm$ depends only on the number of points. Regardless of dimension, this construction and division of $\bdy \cnm$ into strata is functorial in the sense that

\begin{theorem}[\cite{newkey119}]\label{thm:functor}
An embedding $f\co  M\rightarrow N$ induces an embedding of manifolds with corners called the evaluation map $C_n[f] \co C_n[M]\rightarrow C_n[N]$ that respects the stratifications. 
\end{theorem}

\begin{corollary}\label{cor:smooth}
Let $f\co\R^k\rightarrow \R^k$ be a smooth diffeomorphism. Then the induced map of configuration spaces $\cnf\co
\cnr\rightarrow\cnr$ is also a smooth diffeomorphism (on each face of $\cnr$).
\end{corollary}

\begin{proof} This is an immediate corollary of 
the previous theorem.
\end{proof}

Any pair $\bfp$, $\bfq$ of disjoint points in $\R^k$ has a direction $(\bfp-\bfq)/\norm{\bfp-\bfq}$ associated to it, while every triple of disjoint points $\bfp$, $\bfq$, $\bfr$ has a corresponding distance ratio $\norm{\bfp-\bfq}/\norm{\bfp-\bfr}$. One way to think of the purpose of $\cnm$ is that it extends the definition of these directions and ratios to the boundary.

\begin{theorem} [\cite{newkey119} or \cite{newkey118} Theorem 2.3] \label{thm:configgeom} 
Given $M \subset \R^k$, in any configuration of points $\p \in \cnm$ each pair of points $\bfp_i$, $\bfp_j$ has associated to it a well-defined unit vector in $\R^k$ giving the direction from $\bfp_i$ to $\bfp_j$. If the pair of points project to the same point $\bfp$ of $M$, this vector lies in $T_\bfp M$. 

Similarly, each triple of points $\bfp_i$, $\bfp_j$, $\bfp_k$ has associated to it a well-defined scalar in $[0,\infty]$ corresponding to the ratio of the distances $\norm{\bfp_i - \bfp_j}$ and $\norm{\bfp_i - \bfp_k}$. If any pair of $\{\bfp_i, \bfp_j,\bfp_k\}$ projects to the same point in $M$ (or all three do), this ratio is a limiting ratio of distances.

The functions $\pi_{ij}$ and $s_{ijk}$ are continuous on all of $\cnm$ and smooth on each face of $\bdy\cnm$.
\end{theorem}

\section{Special Submanifolds of Configuration Spaces}

We are interested in three special submanifolds of particular configurations defined by geometric constraints. First, we consider the configuration space of points on a curve. 
\begin{definition}
Let $\gamma$  be a $C^\infty$-smooth embedding of $S^1$ in $\R^k$, with $\cng\co\cns \rightarrow \cnr$  the evaluation map on configuration spaces. We abuse notation by using~$\gamma$ to mean either the embedding or its image in~$\R^k$. Similarly, we use $\cng$ to mean either the evaluation map or its image --- the compactified configuration space of $n$ points on the simple closed curve $\gamma(S^1)\in \R^k$. 
\end{definition}

By \thm{functor} we know that  $\cng$ is a submanifold of $\cnr$ and $\bdry\cng\subseteq \bdry\cnr$ with the stratifications respected.  The coordinates for $\cng$ are similar to those described in \thm{configgeom}, as they are the image of the coordinates under $\gamma\co  S^1\rightarrow \R^k$. Volic~\cite{MR2300426} and Budney {\it et al.}~\cite{newkey118} have detailed descriptions of the coordinates for codimension 1 strata. To give an example, observe that the map $C_n[\gamma]$ takes $(\bfp_1,\dots,\bfp_n)\in\cnso$  to $(\gamma(\bfp_1),\dots,\gamma(\bfp_n))\in \cnr$. If we consider the stratum where say $\bfp_1$, $\bfp_2$ and $\bfp_3$ degenerate to a point $\q$ in $\cns$, then $\q$ is a configuration of $n-3+1=n-2$ points plus the $\pij$ and $\sijk$ information for $\bfp_1$, $\bfp_2$ and $\bfp_3$. In $\cnr$ we get a configuration of $n-2$ points on  $\gamma$ plus the directions of approach of the colliding~$\gamma(\bfp_i)$ and the relative distances $s_{123}$, $s_{312}$, and so forth. The $\pi_{ij}$ are unit tangent vectors to $\gamma$. If $\bfp_1$ and $\bfp_3$ approach $\bfp_2$ equally from opposite sides, then in the limit $\norm{\bfp_1 - \bfp_2} + \norm{\bfp_2 - \bfp_3} = \norm{\bfp_1 - \bfp_3}$, so the $s_{ijk}$ obey the relations
\begin{equation*}
1 + s_{231} = s_{132}, \quad s_{213} + 1 = s_{312}, \quad s_{123} + s_{321} = 1.
\end{equation*}

In $\cns$ the values of $\pij$ are in $S^0$ and are mapped to $S^1$ by $\cng$.
Thus, while the exact values of the unit tangent vectors $\pij$ and $\pi_{ji}$ are unknown for two colliding points on $\gamma$, they must differ by $\pi$. 

In the case of the circle, the cyclic ordering of points along $S^1$ determines $(n-1)!$ connected components of $C_n[S^1]$. Note that some strata are empty in the boundary of each connected component of $C_n[S^1]$. For instance, in the component of $\cfs$ where points $\bfp_1$, $\bfp_2$, $\bfp_3$ and $\bfp_4$ occur in order along $S^1$, if $\bfp_1$ and $\bfp_3$ come together, either $\bfp_2$ or $\bfp_4$ must collapse to the same point. Thus the stratum $(13)$ is empty on the boundary of this component. We will focus on one of these connected components:

\begin{definition} 
\label{def:ccng}
Let $\ccng$ denote the component of $\cng$ where the order of the points $\bfp_1, \dots, \bfp_n$ matches the cyclic order of these points along $\gamma$ according to the given parametrization of $\gamma$.
\end{definition}


We now consider another submanifold -- this one with a more interesting structure.

\begin{definition}
Let the subset of {\em square-like quadrilaterals} $\slq$ for $k=2$ be the subspace of squares in $\R^2$, and 
for $k > 2$, the subset of $\cfr$ where $s_{124} = s_{231} = s_{342} = 1$ and $s_{132} - s_{241} = 0$.  That is, $\slq$ is  the space of quadrilaterals in $\R^k$ with equal sides and equal diagonals. 
\label{def:slq}
\end{definition}

\begin{proposition}
\label{prop:slq is submanifold of r2}
The space $\slq \cap \cfro$ is an orientable submanifold of $\cfro$, and the (point-set) boundary of $\slq$ satisfies  
$\bdy\slq \subset \bdy \cfr$.
\end{proposition}

\begin{proof}
Let $\p=((\bfp_1,\bfp_2,\bfp_3,\bfp_4),\alpha(\p))$ be a point  in $\cfr$, and consider the mapping $g\co \cfr \to \R^4$ given by
\begin{align*}
g(\p) &= (s_{124}^2, s_{231}^2, s_{342}^2, s_{132}^2 - s_{241}^2) \\
&= \left( \frac{|\bfp_1-\bfp_2|^2}{|\bfp_1 - \bfp_4|^2},\, \frac{|\bfp_2-\bfp_3|^2}{|\bfp_1 - \bfp_2|^2},\, \frac{|\bfp_3-\bfp_4|^2}{|\bfp_2 - \bfp_3|^2},\,
\frac{|\bfp_1-\bfp_3|^2}{|\bfp_1 - \bfp_2|^2} - \frac{|\bfp_2-\bfp_4|^2}{|\bfp_2 - \bfp_1|^2} \right).
\end{align*}
This mapping is smooth and $\slq$ is the preimage of the point $(1,1,1,0)$. We show
that 
$$dg:T_{\p}\cfro\rightarrow T_{g(\p)}\R^4$$ is onto at points $\p\in\slq$ by showing $dg$ has four linearly independent rows.  We denote a tangent vector at $\p$ by $\cv = \bfv(\p)=(\bfv_1,\bfv_2,\bfv_3,\bfv_4)$, where each $\bfv_i$ is a tangent vector at $\bfp_i$. (Here we suppress the $\alpha(\bfp)$ information on the strata.)

Let $\Delta \bfp_1$ denote a vector at $\bfp_1$ as in Figure~\ref{fig:transversalityvariation}, define $\cv_1 = (\Delta \bfp_1, 0,0,0)$ and consider
\begin{figure}[t]
\begin{overpic}{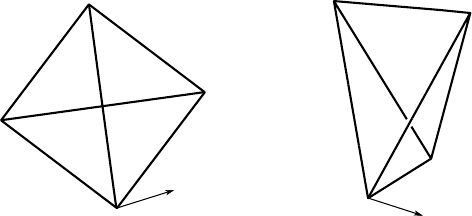}
\put(23,-1.5){$\bfp_1$}
\put(44,25){$\bfp_2$}
\put(19,46.5){$\bfp_3$}
\put(-5,19){$\bfp_4$}
\put(37,4){$\Delta \bfp_1$}

\put(90,-2){$\Delta \bfp_1$}
\put(75,1.5){$\bfp_1$}
\put(101,42){$\bfp_2$}
\put(92,10){$\bfp_3$}
\put(65.5,44.5){$\bfp_4$}

\end{overpic} 
\caption{This figure shows the general situation where a vertex $\bfa$ of a quadrilateral in $\slq$ is varied. On the left, we see the case in the plane, where every quadrilateral in $\slq$ is really a square. On the right, we see the general (space) case, where the quadrilaterals in $\slq$ form a class of special tetrahedra. We compute the corresponding variation of edgelengths, and of the values of the function which we use to define the space of square-like quadrilaterals, in the proof of Proposition~\ref{prop:slq is submanifold of r2}.}
\label{fig:transversalityvariation}
\end{figure}
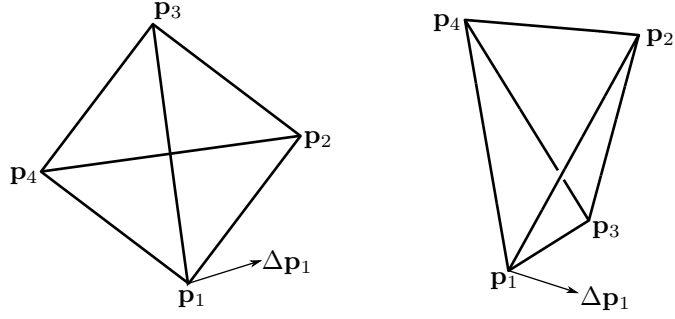
$$dg_{\cv_1} (\bfp_1,\bfp_2,\bfp_3,\bfp_4) = \lim_{|\Delta \bfp_1| \to 0} \frac{g(\bfp_1 + \Delta \bfp_1, \bfp_2, \bfp_3, \bfp_4)
- g(\bfp_1\bf,\bfp_2,\bfp_3,\bfp_4)}{|\Delta \bfp_1|}.$$
To compute the limit, let us consider a typical quotient term involved:
\begin{multline*}
\frac{|\bfp_1 - \bfp_2 + \Delta \bfp_1|^2}{|\bfp_1 - \bfp_4 + \Delta \bfp_1|^2} - \frac{|\bfp_1-\bfp_2|^2}{|\bfp_1-\bfp_4|^2} \\
=\frac{|\bfp_1 - \bfp_2 + \Delta \bfp_1|^2 |\bfp_1-\bfp_4|^2 - |\bfp_1 - \bfp_4 + \Delta \bfp_1|^2 |\bfp_1-\bfp_2|^2}{|\bfp_1 - \bfp_4 + \Delta \bfp_1|^2 |\bfp_1-\bfp_4|^2} \\
= \frac{2 |\bfp_1 -\bfp_4 |^2 (\bfp_1-\bfp_2)\cdot \Delta \bfp_1 - 2 |\bfp_1 - \bfp_2|^2 (\bfp_1-\bfp_4)\cdot \Delta \bfp_1}{|\bfp_1 - \bfp_4 + \Delta \bfp_1|^2 |\bfp_1-\bfp_4|^2}\\
+ \frac{|\bfp_1-\bfp_4|^2 |\Delta \bfp_1|^2 - |\bfp_1-\bfp_2|^2|\Delta \bfp_1|^2}{|\bfp_1 - \bfp_4 + \Delta \bfp_1|^2 |\bfp_1-\bfp_4|^2}.
\end{multline*}
Next divide by $|\Delta \bfp_1|$. We can ignore terms in the numerator 
with $|\Delta \bfp_1|^2$ because they will vanish
in the limit. We rearrange to get:
$$2|\bfp_1 - \bfp_2| \frac{\ds\frac{\bfp_1-\bfp_2}{|\bfp_1-\bfp_2|} \cdot \frac{\Delta \bfp_1}{|\Delta \bfp_1|} - \frac{|\bfp_1-\bfp_2|}{|\bfp_1-\bfp_4|} \frac{\bfp_1-\bfp_4}{|\bfp_1-\bfp_4|}
\cdot \frac{\Delta \bfp_1}{|\Delta \bfp_1|}}{|\bfp_1-\bfp_4 + \Delta \bfp_1|^2}.$$
Taking the limit as $|\Delta \bfp_1| \to 0$, we get:
\begin{multline*}
\lim_{|\Delta \bfp_1| \to 0} \frac1{\Delta \bfp_1} \left(
\frac{|\bfp_1 - \bfp_2 + \Delta \bfp_1|^2}{|\bfp_1 - \bfp_4 + \Delta \bfp_1|^2} - \frac{|\bfp_1-\bfp_2|^2}{|\bfp_1-\bfp_4|^2}\right) =\\
\frac2\ell
(\cos \angle(\Delta \bfp_1, \bfp_1\bfp_2) - \cos \angle (\Delta \bfp_1, \bfp_1\bfp_4)),
\end{multline*}
where $\ell = |\bfp_1-\bfp_2| = |\bfp_2-\bfp_3| = |\bfp_3-\bfp_4| = |\bfp_1-\bfp_4|$, and $\angle(\Delta \bfp_1, \bfp_1\bfp_2)$ is the angle between vector $\Delta \bfp_1$ and the vector given by $\bfp_1\bfp_2 =\bfp_1 - \bfp_2$.

Similar computations give an explicit form to $d{g}$; suppose
$m = |\bfp_1-\bfp_3| = |\bfp_2 - \bfp_4|$. Then
\begin{multline*}
d{g}_{\cv_1}(\bfp_1,\bfp_2,\bfp_3,\bfp_4) = \frac{2}{\ell} \left( \cos \angle (\Delta \bfp_1, \bfp_1\bfp_2) - \cos
\angle (\Delta \bfp_1, \bfp_1\bfp_4),\right. \\
\left. -\cos \angle (\Delta \bfp_1, \bfp_1\bfp_2), 0, \frac{m}{\ell} \cos \angle (\Delta \bfp_1, \bfp_1\bfp_3) \right).
\end{multline*}
Since the angles made by $\Delta \bfp_1$ and the sides and diagonals of a given quadrilateral cannot be chosen so that 
all cosines involved vanish at once, $d{g}_{\cv_1}$ does not vanish on $\slq$.

Analogous variations $\cv_2, \cv_3$, and $\cv_4$ at $\bfp_2$, $\bfp_3$, and $\bfp_4$ respectively, lead to the following expressions:
\begin{align*}
d{g}_{\cv_2} &= \frac{2}{\ell} \left( \cos \angle (\Delta \bfp_2, \bfp_2\bfp_1), 
\cos \angle (\Delta \bfp_2, \bfp_2\bfp_3) - \cos \angle (\Delta \bfp_2, \bfp_2\bfp_1), \right.\\ & \hspace{2.225in} \left. -\cos \angle (\Delta \bfp_2, \bfp_2\bfp_3), -\frac{m}{\ell} \cos \angle (\Delta \bfp_2, \bfp_2\bfp_4)\right)\\
d{g}_{\cv_3} &= \frac{2}{\ell} \left(0, \cos \angle (\Delta \bfp_3, \bfp_3\bfp_2), \cos \angle (\Delta \bfp_3, \bfp_3\bfp_4) -\cos \angle (\Delta \bfp_3, \bfp_3\bfp_2), \right.\\
&\hspace{2.225in}\left.\frac{m}{\ell}\cos \angle (\Delta \bfp_3, \bfp_3\bfp_1)\right)\\
d{g}_{\cv_4} &= \frac{2}{\ell} \left(-\cos \angle (\Delta \bfp_4, \bfp_4\bfp_1), 0, \cos \angle (\Delta \bfp_4, \bfp_4\bfp_3), 
-\frac{m}{\ell} \cos\angle (\Delta \bfp_4, \bfp_4\bfp_2)\right).
\end{align*}
After some elementary row operations, one finds that
carefully chosen variations at $\bfp_1$, $\bfp_2$, $\bfp_3$, and $\bfp_4$ will give four linearly
independent vectors at points in $\slq$. It follows from the Preimage Theorem
of \cite{Guillemin:2010ti} that ${g} \transverse (1,1,1,0)$ and the
interior of $\slq$ is a submanifold of $C_4(\R^k)$.

The boundary points in $C_4[\R^k]$ are where the points of a configuration come together,
along with the directions of collision and ratios of the sides. There is no difficulty in the plane, where the ratios in the definition of ${g}$ may be smoothly extended to the boundary. The boundary $\partial \slq$ is contained in the $(1234)$ boundary face of $C_4[\R^2]$, and, in fact, the map ${g}$ is transverse to $(1,1,1,0)$ on this boundary. (For the sake of brevity we have omitted the details.) Thus in this special case, $\slq$ is actually a submanifold with boundary of $\cfrt$, the larger manifold with boundary.

The (pointset) boundary of $\slq$ in $\cfr$ contains both ``infinitesimal'' squares and configurations in the $(13)(24)$ face of $\cfr$, where the diagonals are equal to zero while the sidelengths remain equal and nonzero. Such collisions lead to square-like quadrilaterals that are four-fold covers of an interval. We may certainly extend the map ${g}$ to this face, but here we run into trouble: Since any configuration on the $(13)(24)$ has equal sidelengths and equal diagonals, the map ${g}$ is \emph{not} transverse to $(1,1,1,0)$ when restricted to this boundary face, and our argument does {\em not} show that $\slq$ is a submanifold with boundary of $\cfr$, the larger manifold with boundary. 
\end{proof}

We next state a useful corollary of these detailed computations. Recall (\cite{Guillemin:2010ti}) that if $f \co X \rightarrow Y$ is transverse to $Z \subset Y$ and $Z$ and $X$ are oriented, the orientation on $f^{-1}(Z)$ at $p \in X$ is constructed by appending a positively oriented basis for the ``horizontal'' subspace of $T_p X$ to a basis for the ``vertical'' subspace $T_p f^{-1}(Z)$. The vertical basis is considered positively oriented if the combined basis is a positively oriented basis for $X$. We will be interested later in the free and properly discontinuous action of $\Z/4\Z$ on $\cfr$ and on $\slq$ that cyclically permutes $\bfp_1$, $\bfp_2$, $\bfp_3$ and $\bfp_4$. Let $\mu \colon \cfr \rightarrow \cfr$ be the map corresponding to the generator of $\Z/4\Z$ for this action. It is clear from the definition of $\slq$ that $\mu$ descends to a map from $\slq$ to $\slq$.

\begin{proposition}
The map $\mu$ reverses orientation on both $\cfr$ and $\slq$ if $k$ is odd, and preserves orientation on both $\slq$ and $\cfr$ if $k$ is even.
\label{prop:orientation}
\end{proposition}

\begin{figure}[th]
\hfill
\begin{minipage}[b]{1.5in}
\begin{overpic}[height=1.5in]{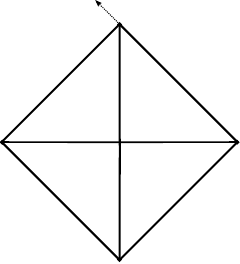}
\put(50,-2){$\bfp_1$}
\put(92,43){$\bfp_2$}
\put(49,93){$\bfp_3$}
\put(-9,43){$\bfp_4$}
\end{overpic} \\

\vspace{0.12in}
\end{minipage}
\hfill
\begin{overpic}[height=1.65in]{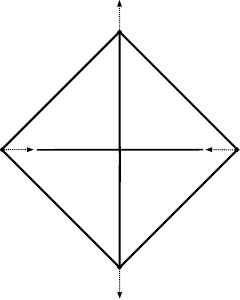}
\put(42,91){$\bfp_3$}
\put(42,8){$\bfp_1$}
\put(81,46.5){$\bfp_2$}
\put(-8,46.5){$\bfp_4$}
\end{overpic}
\hfill
\hphantom{.}

\caption{Two tangent vectors to a configuration $\p$ in $\slq \subset \cfr$ which forms a planar square. For the tangent vector shown at left, the directional derivatives of $\norm{\bfp_1 - \bfp_3}$ and $\norm{\bfp_2 - \bfp_3}$ are positive while the directional derivatives of all other lengths shown vanish. Clearly, we may construct a similar tangent vector at each vertex to increase any given edgelength and corresponding diagonal length while leaving all other lengths unchanged to first order. On the right, we see a tangent vector where the directional derivative of $\norm{\bfp_1 - \bfp_3}$ is positive, the directional derivative of $\norm{\bfp_2 - \bfp_4}$ is negative, and the directional derivatives of all other lengths vanish.
\label{fig:slqmotions}}
\end{figure}

\begin{proof}
 We first note that $T_{\p}\slq\subset T_{\p}\cfr$, and recall that a tangent vector at $\p$ is denoted by $\cv = \bfv(\p)=(\bfv_1,\bfv_2,\bfv_3,\bfv_4)$, where $\bfv_i$ is a tangent vector at $\bfp_i$. 

To prove the proposition, we now construct some specific variations of quadrilaterals in $\slq$ that will behave nicely under the $\Z/4\Z$ action. For squares in the plane, Figure~\ref{fig:slqmotions} shows the construction of two types of tangent vectors to $\cfr$ at $\p$. The first three tangent vectors are of the form $\cu = (0,\bfu_2,0,0)$, $\cv = (0,0,\bfv_3,0)$ and $\cw = (0,0,0,\bfw_4)$. Note $\cv$ is shown at the left in the figure and $\bfv_3$ is perpendicular to ${\bfp_3\bfp_4}$. Assume that $\p\in\slq$ has $l=|\bfp_2-\bfp_1|=|\bfp_3-\bfp_2|=|\bfp_4-\bfp_3|=|\bfp_1-\bfp_4|$ and $l/\sqrt{2}=|\bfp_1-\bfp_3|=|\bfp_2-\bfp_4|$. As shown in the figure, we can arrange to have 
\begin{align*}
D_{\cu} \norm{\bfp_1 - \bfp_2} &= +\ell/2, \quad  D_{\cu} \norm{\bfp_2 - \bfp_4} = \ell/2\sqrt{2}, \quad \text{other directional derivs of lengths } = 0, \\
D_{\cv} \norm{\bfp_2 - \bfp_3} &= +\ell/2, \quad  D_{\cv} \norm{\bfp_1 - \bfp_3} = \ell/2\sqrt{2}, \quad \text{other directional derivs of lengths } = 0, \\
D_{\cw} \norm{\bfp_3 - \bfp_4} &= +\ell/2, \quad  D_{\cw} \norm{\bfp_2 - \bfp_4} = \ell/2\sqrt{2}, \quad \text{other directional derivs of lengths } = 0.
\end{align*}
The fourth tangent vector $\x$ is shown at the right in Figure~\ref{fig:slqmotions} and has  
\begin{equation*}
D_{\x} \norm{\bfp_1 - \bfp_3} = +\ell/2\sqrt{2}, D_{\x} \norm{\bfp_2 - \bfp_4} = -\ell/2\sqrt{2},  \text{other directional derivs of lengths } = 0.
\end{equation*}
Working out the directional derivatives of $s_{124}^2$, $s_{231}^2$, $s_{342}^2$, and $s_{132}^2 - s_{241}^2$ in these directions, we see that $Dg$ restricted to the span of $\cu$, $\cv$, $\cw$, and $\x$ looks like the matrix:
\begin{equation*}
Dg = \begin{pmatrix}
1  & 0  & 0 & 0 \\
-1 & 1  & 0 & 0 \\
0  & -1 & 1 & 0 \\
\ast  & \ast  & \ast & 2 \\
\end{pmatrix}
\end{equation*}
where the $\ast$ entries represent nonzero values that we don't need to compute. 

Now we make a similar construction for nonplanar configurations in $\slq$. Assume the square-like quadrilateral $\p\in\slq$ has sides of length $\ell=|\bfp_{1}-\bfp_2|$ etc., and diagonals have length $m=|\bfp_1-\bfp_3|=|\bfp_2-\bfp_4|$. Consider the situation shown in Figure~\ref{fig:slqmotions2}. Let us focus on edge $\bfp_2 \bfp_3$ for convenience. At $\bfp_3$ the plane determined by $\bfp_1$, $\bfp_4$, and $\bfp_3$ has normal vector, say $\bf{n}$. Consider the tangent vector ${\cn}=(0,0,{\bf n},0)$. Since $\bf{n}$ is perpendicular to vectors $\bfp_4 \bfp_3$ and $\bfp_1 \bfp_3$, the directional derivatives of the lengths of these edges in this direction are zero. On the other hand, since the tetrahedron is not a planar square, $\bfp_2 \bfp_3$ is not in the plane normal to ${\bf n}$, so the directional derivative of $\norm{\bfp_2 - \bfp_3}$ is nonzero. We can now find some scalar multiple $\cv_3$ of $(0,0,{\bf n},0)$ so that $D_{{\cv}_3} \norm{\bfp_2 - \bfp_3} = \ell/2$. This implies that $Dg(\cv_3) = (0,1,-1,0)$.  

We can make a similar argument at vertex $\bfp_1$. Let ${\bf n}$ be a normal vector of the $\bfp_1 \bfp_2 \bfp_3$ plane and find $\cv_1$ parallel to ${(\bf{n}},0,0,0)$ so that $D_{{\cv}_1} \norm{\bfp_1 - \bfp_4} = -\ell/2$ and $Dg(\cv_1) = (1,0,0,0)$. A similar argument at $\bfp_4$ yields a vector $\cv_4$ with $D_{{\cv}_4} \norm{\bfp_3 - \bfp_4} = \ell/2$ while preserving all other edgelengths to first order. Scaling appropriately, we can arrange to have $Dg(\cv_4) = (0,0,1,0)$. 

As shown in Figure~\ref{fig:slqmotions2} at right, we can also find a tangent direction $\cw$ so that $D_{\cw} \norm{\bfp_3 - \bfp_1} = \ell^2/2m$ while the directional derivatives of all other edgelengths vanish. This choice gives  $Dg(\cw) = (0,0,0,1)$. Taken together, we have constructed a subspace of $T_{\p} \cfr$ given by $\operatorname{Span}(\cv_1,\cv_3,\cv_4,\cw)$ on which 
\begin{equation*}
Dg = \begin{pmatrix}
1  & 0  & 0 & 0 \\
0  & 1  & 0 & 0 \\
0  & -1 & 1 & 0 \\
0  & 0  & 0 & 1 \\
\end{pmatrix}
\end{equation*}

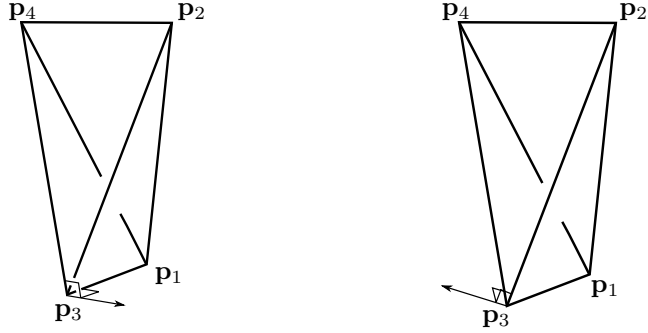
\begin{figure}
\hfill
\begin{overpic}[height=1.5in]{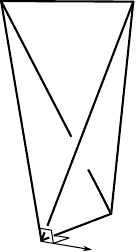}
\put(-4,102){$\bfp_4$}
\put(47,10){$\bfp_1$}
\put(55,102){$\bfp_2$}
\put(12,-3){$\bfp_3$}
\end{overpic}
\hfill
\begin{overpic}[height=1.5in]{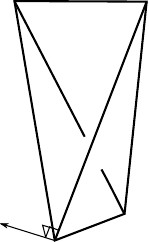}
\put(0,102){$\bfp_4$}
\put(52,5.5){$\bfp_1$}
\put(62,102){$\bfp_2$}
\put(14,-5){$\bfp_3$}
\end{overpic}
\hfill
\hphantom{.}
\\

\caption{Two types of motions of a quadrilateral with equal sides and equal diagonals in $\R^k$. Such a quadrilateral is always a tetrahedron which projects to a square along the axis joining the midpoints of the diagonals. The motion on the left increases $\norm{\bfp_2 - \bfp_4}$ to first order while preserving all other edgelengths. The motion on the right decreases the length $\norm{\bfp_3 - \bfp_4}$ to first order, while preserving all other edgelengths to first order.}
\label{fig:slqmotions2}
\end{figure}

Using these bases, we can now compute the effect of the $\Z/4\Z$ action $\mu$ on the orientation of $\cfr$ and $\slq$. 
First, observe that the tangent space to $\cfr$ contains of four copies of $T\R^k$ and that reordering these from $(1,2,3,4)$ to $(2,3,4,1)$ requires $3k^2$ swaps of basis elements. Thus $\mu$ is orientation preserving or reversing on $\cfr$ as $k$ is even or odd. 

Now take any positively oriented basis $B$ for $T_{\p} \slq$ and extend it by a basis $B'$ so that $Dg$ maps $\Span (B')$ onto the tangent space to $\R^4$ in such a way that the image of $\Span (B')$ is positively oriented with respect to the orientation of $\R^4$. We want to know whether $\mu(B)$ is positively oriented. We know that the combined basis $\mu(B,B')$ is positively or negatively oriented in $\cfr$ as $k$ is even or odd. It remains to show that $Dg$ maps $\Span(\mu(B'))$ onto the tangent space for $\R^4$ so that the image is positively oriented. This comes down to an explicit calculation of determinants. 

For a planar configuration $\p\in\slq$, we use the basis $\cu$, $\cv$, $\cw$, $\x$ constructed above. We can compute that, on the space $\Span(d\mu_{\p}(\cu), d\mu_{\p}(\cv), d\mu_{\p}(\cw), d\mu_{\p}(\x))$, we have:
\begin{equation*}
Dg = \begin{pmatrix}
-1  &  1  & 0 & 0 \\
0  &  -1  &  1 & 0 \\
0 &  0  &  -1 & 0 \\
\ast  &  \ast  &  \ast & -2 \\
\end{pmatrix},
\end{equation*}
where again, $\ast$ represents a value we don't need to compute. This is a matrix of positive determinant, as desired. For a non-planar configuration in $\slq$, we use the basis $\cv_1$,$\cv_3$,$\cv_4$,$\cw$ constructed above and compute that, on $\Span(d\mu(\cv_1),d\mu(\cv_3),d\mu(\cv_4),d\mu(\cw))$, we have
\begin{equation*}
Dg = \begin{pmatrix}
0  &  1  & 0 & 0 \\
0   &  -1  &  1 & 0 \\
1   &  0  &  -1 & 0 \\
0   &  *  &  0 & -1 \\
\end{pmatrix}.
\end{equation*}
Again, this is a matrix of positive determinant, as desired.
\end{proof}

The the third interesting submanifold of configuration space is the configuration space of top-dimensional simplices with edgelengths in a given ratio.

\begin{definition}
Suppose we have a ratio of $\binom{k+1}{2}$ positive distances. It will be convenient to denote this ratio $R$ by $(k+1)^2$ coefficients $d_{ij}$ where $d_{ii} = 0$ and $d_{ji} = d_{ij}$ (these are not unique). We will call such a ratio a \emph{simplex distance ratio}. The space of configurations in $C_{k+1}[\R^k]$ given by points $\p=(\bfp,\alpha(\bfp))$ with $s_{ijk}(\bfp) = d_{ij}/d_{ik}$ will be denoted $\simpR$. This simplex distance ratio will be called \emph{constructible} if $\simpR$ is nonempty.
\label{def:constructible}
\end{definition}

The theory of distance geometry allows us to decide which ratios are constructible by a simple calculation:

\begin{theorem}[Cayley-Menger Theorem~\cite{Blumenthal:1943vx}, (cf. \cite{berger1987geometry}, Section 9.7)]
A simplex distance ratio $R = \{d_{ij}\}$ is constructible $\iff$ the Cayley-Menger determinant:
\begin{equation*}
D(d_{11}, \dots, d_{k+1,k+1}) = 
\left|
\begin{matrix}
0 & 1 & 1 & \dots & 1 \\
1 & 0 & d_{12}^2 & \dots & d_{1,k+1}^2 \\
1 & d_{21}^2 & 0 & \dots & d_{2,k+1}^2 \\
\vdots & \vdots & \vdots &  & \vdots \\
1 & d_{k+1,1}^2 & d_{k+1,2}^2 & \dots & 0 \\
\end{matrix}
\right|
\end{equation*}
is non-negative. In fact, if $d_{ij} = |\bfp_i - \bfp_j|$ for $\bfp_1,\dots,\bfp_{k+1}\in\R^k$, the volume $V$ of the simplex with vertices $\bfp_1, \dots, \bfp_{k+1}$ obeys 
\begin{equation*}
V^2(d_{11}, \dots, d_{k+1,k+1}) = \frac{(-1)^{k+1}}{2^k(k!)^2} D(d_{11}, \dots, d_{k+1,k+1}).
\end{equation*}
\end{theorem}
If we fix the simplex distance ratio $R$, then we note that when the Cayley-Menger determinant is positive, the configurations $\p\in\simpR$ consist of similar copies of the same simplex. 

The Cayley-Menger determinant generalizes standard facts in triangle geometry: for instance, for a triangle with side lengths $a$,$b$, and $c$ we can write this determinant explicitly as 
\begin{equation*}
D(a,b,c) = a^4-2 a^2 b^2-2 a^2 c^2+b^4-2 b^2 c^2+c^4 = -(a+b+c)(a+b-c)(a-b+c)(-a+b+c).
\end{equation*}
and conclude that 
\begin{equation*}
\operatorname{Area}(a,b,c)^2 = \frac{1}{16} (a+b+c)(a+b-c)(a-b+c)(-a+b+c).
\end{equation*}
which is Heron's formula for the area of the triangle. We can see the triangle inequality, (a criteria for constructability of a triangle), in this formula: the sign of the squared area would be negative if and only if one of the side lengths was greater than the sum of the other two. Our previous ``degenerate'' ratios for square-like quadrilaterals correspond to cases where one of the side lengths is equal to the other two: in such a case the Cayley-Menger determinant (and the volume of the simplex) vanish. This motivates the following:
\begin{definition}
A simplex distance ratio $R = \{d_{ij}\}$ is \emph{degenerate} if $D(d_{11},\dots,d_{k+1,k+1}) = 0$.
\end{definition}

We can characterize the space $\simpR$ in a useful way:
\begin{proposition}
If $R$ is a constructible, nondegenerate simplex distance ratio, then $\simpR$ is a submanifold with boundary of $C_{k+1}[\R^k]$,
diffeomorphic to $O(k) \cross \R^{k} \cross [0,\infty)$,
and $\bdy \simpR \subset (1 \cdots k+1) \subset \bdy C_{k+1}[\R^k]$. 
\label{prop:simpR is submanifold}
\end{proposition}

Each configuration $\p$ in $\simpR$ is a similar copy of a single simplex, while the boundary consists of ``infinitesimal'' copies of the same simplex. 


\begin{proof} 
To construct a map $f:\simpR \rightarrow O(k) \cross \R^k \cross [0,\infty)$ explicitly, take a point $\p=(\bfp,\alpha(\bfp))$ in $\simpR$ where $\bfp=(\bfp_1,\dots,\bfp_{k+1})$, and consider the matrix  of vectors $A_{\p}=\begin{bmatrix} \pi_{12} &\pi_{13} &\hdots& \pi_{1(k+1)} \end{bmatrix}$.  Since the simplex distance ratio is nondegenerate, the Cayley-Menger theorem tells us that the column vectors of $A_{\p}$ are $k$ linearly independent vectors in $\R^k$. The Gram-Schmidt process provides a smooth map taking any such configuration to a matrix in $O(k)$. We denote this process by $ \operatorname{GS}(A_{\p})$. It is easy to see that the Gram-Schmidt process obeys the equivariance relation $\operatorname{GS}(B \cdot A_{\p}) = B\cdot\operatorname{GS}(A_{\p})$ for any matrix $B \in O(k)$.  

We can now define our map to be $f (\p):= \operatorname{GS}(A_{\p}) \cross \bfp_1 \cross |\bfp_1 - \bfp_2|$. By the equivariance property above, and the since the action of $O(k)$ on a nondegenerate simplex $\p$ has no fixed points, this is a smooth bijection from $\simpR$ to $O(k) \cross \R^k \cross [0,\infty)$. Note that when $|\bfp_1-\bfp_2|=0$ (since the ratios of all pairwise distances are fixed) the simplex $\p$ must be ``infinitesimal", and $\p$ must lie in the $(1 \cdots k+1)$ stratum. Indeed we find $\bdy \simpR \subset (1 \cdots k+1)\subset \bdy C_{k+1}[\R^k]$.

To show that $f$ is a diffeomorphism, we must consider the differential of the map and prove that it has no kernel. So consider a variation $\cv$ of $\p$. If it moves $\bfp_1$, then $Df(\cv)$ has a nonzero component in the $\R^k$ coordinates. Noting that the action of $\cv$ on $\pij$ changes no $\sijk$, then if $\cv$ changes any pairwise distance between vertices to first order,  it changes the pairwise distance between vertices $\bfp_1$ and $\bfp_2$ to first order, and hence $Dg(\cv)$ has a nonzero component in the $[0,\infty)$ coordinate. So suppose that $\cv$ changes no $\pij$. By Alexandrov's theorem on rigidity of convex polyhedra (see Theorem~25 of~\cite{MR2127379}) this implies that $\cv$ generates a motion in $O(k)$. Differentiating the equivariance relation above completes the proof.
\end{proof} 

By affine independence, $\simpR$ deformation retracts to $O(k)$, and so it has the homology of $O(k)$. Let this projection be denoted $\pi \colon \simpR \rightarrow O(k)$. Copies of the simplex in $\simpR$ that share an orientation form a connected component of $\simpR$ diffeomorphic to $SO(k) \cross \R^{k} \cross [0,\infty]$; we will denote the configurations $\p = (\bfp_1, \dots, \bfp_{k+1})$ in $\simpR$ where the matrix with columns $\pi_{1(k+1)}, \dots, \pi_{k(k+1)}$ has positive determinant by $\simpR^+$.

\section{Configuration Spaces and Transversality }

In this section, we prove a transversality ``lifting property'' for compactified configuration spaces: The submanifold of configurations of points on a smoothly embedded submanifold $M$ of $\R^k$ may be made transverse to any submanifold $Z$ of the configuration space of points in $\R^k$ by an arbitrarily small variation of $M$, as long as the two submanifolds of configuration space are boundary-disjoint. This is a useful technique and parts of it have been proved before. For instance, Budney et al.~\cite{newkey118} prove a special case of this result. 
We will show that a general form of this result may be obtained easily from the Multijet Transversality Theorem~(\cite{Golubitsky:1974iu}, Theorem 4.13).

We begin by recalling some details about the construction of jet space and the Whitney $C^\infty$ topology on mappings. Then we will state the multijet transversality theorem and show that our desired result on configuration space transversality follows. 

\begin{definition}
\label{def:jetspace}
Let $M$ and $N$ be smooth manifolds, and $f$ be a smooth function $f \co M \rightarrow N$. The {\bf space of $0$-jets} $J^0(M,N) = M \cross N$. The {\bf $0$-jet of $f$} is the function $j^0 f \co M \rightarrow J^0(M,N)$ given by $j^0 f(\bfp) = (\bfp,f(\bfp))$.
\end{definition}

It is a standard fact that jet space $J^0(M,N)$ is a smooth manifold. Further, $0$-jet spaces may be extended to $k$-jet spaces by an inductive procedure involving taking successive derivatives. We won't need higher jet spaces here, so we refer the interested reader to \cite{Golubitsky:1974iu} for details. We can extend the definition of jet space to a space of $n$-fold \emph{multijets} as follows.

\begin{definition}
\label{def:multijets}
The {\bf $n$-fold $0$-multijets} $J_n^0(M,N)$ are the configuration space $C_n(J^0(M,N))$. Given a smooth function $f \co M \rightarrow N$, there is a natural smooth map $j^0_n f \co \cnmo \rightarrow J_n^0(M,N)$ given by
\begin{equation*}
j^0_n f(\p) = (j^0 f(\bfp_1), \dots, j^0 f(\bfp_n)).
\end{equation*}
\end{definition}
If this definition seems a bit puzzling, recall that the jet $j^0 f (\bfp_i)$ includes the location $\bfp_i$ as part of its data, so there is no danger of ``collisions'' in the tuple $(j^0 f(\bfp_1), \dots, j^0 f(\bfp_n))$ because the $\bfp_i$ are distinct by assumption. Notice also that while the space $\cnmo$ includes much more data than the $\bfp_i$, all that additional data is determined uniquely by the $\bfp_i$ so the extra information is basically irrelevant here. 

We can now state the theorem we need:
\begin{theorem}[0-Multijet Transversality Theorem, \cite{Golubitsky:1974iu} Theorem 4.13]
Let $M$ and $N$ be smooth manifolds and let $Z$ be a submanifold of $C_n(J^0(M,N))$. Let 
\begin{equation*}
T_Z = \left\{ f \in C^\infty(M,N) \mid j_n^0 f \transverse Z \right\}.
\end{equation*}
Then $T_Z$ is $C^m$-dense in $C^\infty(M,N)$ for any $m$. In fact, if $Z$ is compact, then $T_Z$ is $C^\infty$ open in $C^\infty(M,N)$. 
\label{thm:multijet}
\end{theorem}
We note that the theorem is actually a bit stronger than the version we have stated, as it shows that $T_Z$ is a \emph{residual} set, meaning a countable intersection of open dense subsets of $C^\infty(M,N)$. We also note that the topology we're using on $C^\infty(M,N)$ is the (standard) Whitney $C^\infty$ topology. We can now apply this to show:

\begin{theorem}[{\bf Transversality Theorem for Configuration Spaces}]\label{thm:transversality-config}
Assume that $M$ is a compact manifold, smoothly embedded in $\R^k$, with corresponding compactified configuration spaces $\cnm$ and $\cnr$.  Assume that $Z$ is a closed topological space contained in $\cnr$ so that $Z \cap \cnro$ is a submanifold of $\cnro$ and the (set-theoretic) boundary of $Z$ is contained in $\bdry\cnr$ and is disjoint from $\bdry\cnm$. Then there exists a manifold $M'$ which is $C^\infty$ close to $M$ such that $C_n(M') \transverse (Z \cap \cnro)$ inside $\cnro$ and $\bdy Z$ and $\bdy M'$ are disjoint in $\bdy \cnr$.
\end{theorem}

\begin{proof}
Since $M$ is compact, the closed set $\cnm \cap Z$ is also compact. Since this compact set is disjoint from the closed set $\bdy\cnr$, it is separated from $\bdy\cnr$ by some $\epsilon > 0$. Replace $Z$ with its intersection $Z'$ with the interior of the complement of an $\epsilon/2$ neighborhood of $\bdy\cnr$. This $Z'$ is now an open manifold contained in $\cnro$ and remaining a bounded distance from $\bdy\cnr$.

Let $\iota \co M \rightarrow \R^k$ be the inclusion map from $M$ to $\R^k$. We will prove that a $C^\infty$-small modification $\iota'$ of $\iota$ gives $C_n[\iota'(M)]$ that is transverse to $Z$. In the first place, since $M$ is compact, a $C^\infty$ small modification $\iota'$ is still a diffeomorphism onto its image, and hence still a smooth embedding of $M$ into $\R^k$ with image a manifold $\iota'(M) = M'$ which is $C^\infty$ close to $M$. 

Next, since $C_n[-]$ is a continuous map from $C^\infty(M,\R^k)$ to $C^\infty(\cnm,\cnr)$, $C_n[M']$ will be $C^\infty$ close to $C_n[M]$ and hence we can choose the modification of $\iota$ small enough that the intersections of $C_n[M']$ with $Z'$ are at least $(3/4) \epsilon$ from $\bdy \cnr$. This means that they are intersections with the original $Z$ and that $C_n[M'] \transverse Z' \implies C_n[M'] \transverse Z$. 
Since $Z'$ does not approach $\bdy\cnr$, it suffices to show that we can modify $\iota$ so that $C_n(\iota') \co \cnmo \rightarrow \cnro$ is transverse to $Z'$. 

Generally speaking, the $n$-fold $0$-multijet $j^0_n(\iota)$ maps $\cnmo$ into $C_n(J^0(M,\R^k)) = C_n(M \cross \R^k)$; that is, it should map a disjoint collection of points $\bfp_i \in M$ to a disjoint collection of pairs in the form $(\bfp_i,\iota(\bfp_i))$ in $M \cross \R^k$. But since $\iota$ is a diffeomorphism onto its image, it is $1-1$, and the $\iota(\bfp_i)$ are distinct as well as the $\bfp_i$. This means that we can think of such a multijet as a map 
\begin{equation*}
j^0_n \iota \co \cnmo \rightarrow \cnmo \cross \cnro, \text{ where } j^0_n(\p) = (\p,C_n(\iota)(\p)).
\end{equation*}
Since being a diffeomorphism onto the image is a stable property under $C^\infty$ perturbations of a map, we may view the $n$-fold $0$-multijet of any perturbation $\iota'$ of $\iota$ in the same way.

Now define a (relatively open) submanifold of $\cnmo \cross \cnro$ by $\cnmo \cross Z'$. Applying Theorem~\ref{thm:multijet}, we see that there is some map $\iota'$ which is $C^\infty$ close to $\iota$ so that $j^0_n \iota'$ is transverse to $Z$. We claim that this implies $C_n(\iota') \transverse Z'$ and hence completes the proof. This follows easily from the definition of transversality if we consider the commutative diagram below ($\pi$ is projection). 
\begin{equation*}
\begin{diagram}
\cnmo & \rTo^{j^0_n \iota'}  & \cnmo \cross \cnro \\
      & \rdTo_{C_n(\iota')}  & \dTo_\pi \\
      &       & \cnro  \\
\end{diagram}
\end{equation*} 
\end{proof}

\section{Applications}

We have now established that the configuration space of $n$-tuples of points in $\R^k$ can be viewed as a manifold with boundary $C_n[\R^k]$, and that, for any smooth submanifold of $M$ of $\R^k$, there is a proper embedding of $C_n[M] \hookrightarrow C_n[\R^k]$ so that $C_n[M]$ is transverse to $\bdy C_n[\R^k]$. We now specialize to the case where $M$ is a sphere $S^l$ and show any $C^1$ embedding $\gamma$ of $S^l$ in $\R^k$ is $C^1$ close to a smooth embedding $\gamma$ for which $C_n[\gamma']$ is guaranteed to have certain intersections with various ``target'' submanifolds of $\cnr$ defined by geometric conditions. This will prove that a dense set of embeddings of $S^l$ always contain certain inscribed configurations of points.

These applications will follow the same basic pattern:

\begin{itemize}
\item Establish the existence of a transverse intersection between $C_n[S^l]$ and the target submanifold $Z$ inside $\cnr$ for a standard embedding of $S^l$. Compute the homology class of the intersection.
\item Use our transversality theorem to find a smooth embedding $\gamma'$ of $S^l$ which is $C^1$-close to the original embedding $\gamma$ so that $C_n[\gamma'] \transverse Z$. This will require that $\cng$ and $Z$ are boundary-disjoint.
\item Use Haefliger's theorem on smooth embeddings~\cite{Haefliger:1961wr} to find a smooth map $E:S^l \cross I \rightarrow \R^K$ with $E(-,0)$ our standard embedding and $E(-,1) = \gamma'$ (where $K$ may be greater than our original $k$). Lift $E$ to a map $C_n[S^l] \cross I \rightarrow C_n[\R^K]$ by functoriality. Now modify this lifted map using the transversality homotopy theorem to be transverse to $Z$ everywhere.
\item Conclude that the intersections $C_n[S^l] \cap Z$ and $C_n[\gamma'] \cap Z$ are cobordant in $C_n[S^l] \cross I$ and hence that they represent the same homology class in $Z$.
\end{itemize}

We recall Haefliger's result in a form useful to us (actually, his result is stronger). We use this result to deform our standard spheres into the spheres of interest. Generally, such an isotopy must pass through spheres embedded in a higher-dimensional space, as when the spheres are knotted. Since differentiable knotting is stronger than topological knotting and we prefer to work in the differentiable category, we will need even more extra room to work\footnote{With various topological tameness assumptions, it would be enough to pass through $l$-spheres in $\R^{l+3}$ by Zeeman's result on PL-unknotting~\cite{Zeeman:1963ta}, but there seems to be no practical penalty for using the differentiable result as we start and end with a sphere in the original $\R^k$ in any case.}:
\begin{theorem}{{\rm \cite{Haefliger:1961wr}}}
Any two differentiable embeddings of $S^l$ in $\R^k$ are differentiably isotopic through an isotopy in $\R^{K} \supset \R^k$ when $K \geq \max\{4l,k\}$.
\label{thm:isotopy}
\end{theorem}

\subsection{The ``square-peg'' theorem}

We can now prove a version of the square-peg theorem. Recall from Definition~\ref{def:ccng} that $\cfog$ is the submanifold of $4$-tuples on a curve $\gamma$ where the points occur in order according to the orientation of the curve, and from Definition~\ref{def:slq} that $\slq$ is the submanifold of configurations $\p$ of 4 points in $\R^k$ with equal ``sides'' $\norm{\bfp_1 - \bfp_2} = \norm{\bfp_2 - \bfp_3} = \norm{\bfp_3 - \bfp_4} = \norm{\bfp_4 - \bfp_1}$ and equal ``diagonals'' $\norm{\bfp_1 - \bfp_3} = \norm{\bfp_2 - \bfp_4}$.

We will show that when $\cfog \transverse \slq$, the number of intersections is an odd multiple of 4, giving an odd number of inscribed ``squares'' up to cyclic relabeling. We note that when $\cfog$ is not transverse to $\slq$ this count need not be odd, as shown by the examples of Popvassiliev~\cite{2008arXiv0810.4806P}.


\begin{theorem}
For any $C^1$ curve in $\R^k$, there is a $C^1$ -close curve $\gamma$ where 
$$
\cfog \cap \slq = \{ \text{an odd, finite set of inscribed squarelike quadrilaterals} \}.
$$
\label{thm:squarepeg}
\end{theorem}
This theorem is illustrated by the three squares inscribed in an irregular curve shown in Figure~\ref{fig:slq}.
\begin{figure}
\hfill
\includegraphics[width=1.25in]{ghostSlqleft.pdf}
\hfill
\includegraphics[width=1.25in]{ghostSlqcenter.pdf}
\hfill
\includegraphics[width=1.25in]{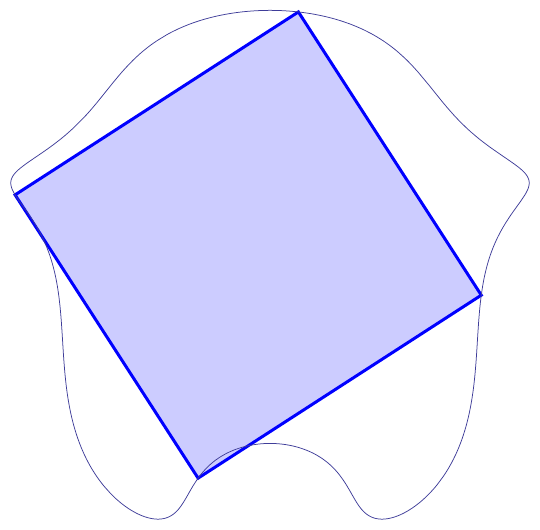}
\hfill
\includegraphics[width=1.25in]{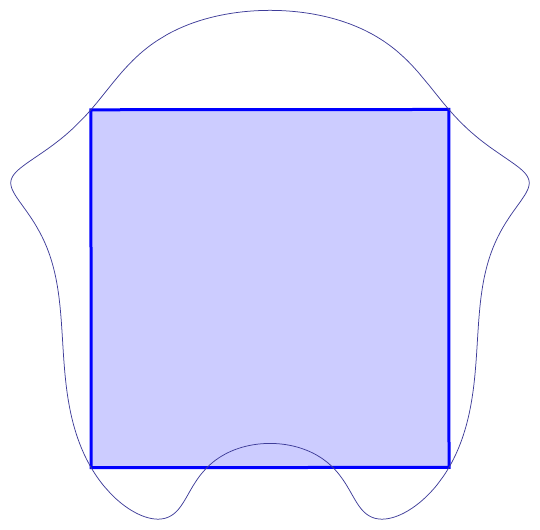}
\hfill
\hphantom{.}
\caption{This picture shows three of the five squares inscribed in an irregular three-lobed curve and two of the three squares inscribed in an irregular ``tooth-shaped'' curve. Since each family shares the vertical flip symmetry of each curve, we show the center (symmetric) square in the second and fourth pictures, while the first and third show half of the asymmetrical squares. While on the left curve the squares are fairly close together, a computer search reveals that they are certainly distinct.  
\label{fig:slq}}
\end{figure}

\begin{proof}
We want to compute the homology class in $H_0(\slq,\Z)$ of the intersection of $\cfog$ and $\slq$ for a transverse intersection. Unfortunately, while $\slq \cap \cfog$ is indeed $0$-dimensional, the intersection represents $0$ in the homology $H_0(\slq;\Z) = \Z$. The essential problem is that a squarelike quadrilateral can be cyclically relabeled in four ways, and it turns out that these relabelings alternate signs in $H_0(\slq;\Z)$. We can fix the problem by identifying these relabelings as a single configuration:

\begin{proposition}
The manifolds $\cfr$, $\cfog$, and $\slq$ share a smooth, free, and properly discontinuous $\Z/4\Z$ action given by cyclically relabeling points in a configuration. 
\begin{itemize}
\item The generator $(\bfp_1,\bfp_2,\bfp_3,\bfp_4) \mapsto (\bfp_2,\bfp_3,\bfp_4,\bfp_1)$ is always orientation-reversing on $\cfog$. It is orientation-reversing on both $\cfr$ and $\slq$ if $k$ is even and orientation preserving on $\cfr$ and $\slq$ if $k$ is odd.
\item The quotient spaces by the action of $\Z/4\Z$, $\qcfr$ and $\qcfog$, are manifolds with boundary and corners, with $\qcfog$ non-orientable and $\qcfr$ orientable as $k$ is odd or even. 
\item The intersection of $\qslq$ with the complement of an $\epsilon$-neighborhood of the boundary face $(13)(24)$ (which is preserved under the action), is a manifold with boundary. It is orientable precisely when $\qcfr$ is.
\end{itemize}
\end{proposition}

\begin{proof}
It is easy to see that this action on $\cfr$ is smooth, free and properly discontinuous and that it descends to a corresponding action on the submanifolds $\cfog$ and $\slq$ (cf. Theorem~4.2 of~\cite{newkey119}). The second point was proved in Proposition~\ref{prop:orientation} when we proved that $\slq$ was a submanifold of $\cfr$. The other points are easy consequences. We note for the third point that the action is actually an isometry on $\cfg$, so it does descend to the $\epsilon$-neighborhood of $(13)(24)$ as needed.
\end{proof}

We now prove:

\begin{proposition}
In $\R^2$, if $\gamma$ is a planar ellipse $\displaystyle \nicefrac{x^2}{a^2} + \nicefrac{y^2}{b^2} = 1$ with $a^2 \neq b^2$, $\qcfog \transverse \qslq$ and the intersection represents a single square.
\end{proposition}

\begin{proof}
We will need a lemma:

\begin{lemma}\label{lem:parallelchords}
Parallel chords meeting an ellipse have midpoints on a line through the center of the ellipse (where the major and minor axes meet).
\end{lemma}
\begin{proof} 
This is true for a circle and preserved under affine mappings.
\end{proof}

We prove that the intersection is a single square. First, if we intersect the ellipse with the lines $y = \pm x$, by symmetry the intersection points form a square. If we parametrize the ellipse by $(x(\theta),y(\theta)) = (a \cos \theta, b \sin\theta)$ we can work out that $\cos^2 \theta = b^2/(a^2 + b^2)$ and $\sin^2\theta = a^2/(a^2 + b^2)$. We prove that this is the only square inscribed in the ellipse.

Suppose $ABCD$ is any square inscribed in the ellipse. Let $M$ denote the midpoint of $AB$
and~$N$ denote the midpoint of $CD$. Then, by \lem{parallelchords}, $MN$ passes
through the center $O$ of the ellipse. Similarly, if $K$ denotes the midpoint of $AD$ and
$L$ the midpoint of $BC$, then $KL$ passes through~$O$. Thus $O$ is also the center of the square.
Parametrize the ellipse by $\theta\mapsto (a\cos \theta, b\sin\theta)$. Then write
\begin{equation*}
A = (a\cos\alpha, b\sin\alpha), \quad B = (a\cos\beta, b\sin \beta).
\end{equation*}
The segment $OM$ is perpendicular to $AB$ and so $\triangle OAM$ and $\triangle OBM$
are congruent and $OA\cong OB$. Thus
\begin{equation*}
a^2 \cos^2\alpha + b^2 \sin^2 \alpha = a^2\cos^2 \beta + b^2\sin^2\beta.
\end{equation*}
This implies $(a^2 - b^2) \cos^2 \alpha = (a^2 - b^2)\cos^2 \beta$ and so, since $a\neq b$, we know $\cos \alpha
= \pm \cos\beta$. Similarly, $\sin \alpha = \pm \sin \beta$. This means that $B$ is the image of $A$ under a symmetry of the ellipse, and since the same argument works \textit{mutatis mutandis} for $C$ and $D$, the square is symmetric under the flip symmetries of the ellipse. There are two types of inscribed quadrilaterals with this symmetry: inscribed rectangles in the form $(\pm x, \pm y)$, and the ``exceptional'' rhombus $\{(\pm a,0),(0,\pm b)\}$. Since $a \neq b$, the only square is our previous set of 4 points $(\pm ab/\sqrt{a^2 + b^2}, \pm ab/\sqrt{a^2+b^2})$.

We now prove that the intersection of $\qslq$ and $\qcfog$ is transverse for the ellipse. We note that $\qcfog$ is always far from the $(12)(34)$ face of $\qcfr$, so $\qslq$ is a manifold at these points. It suffices to prove transversality for $\cfog$ and $\slq$.

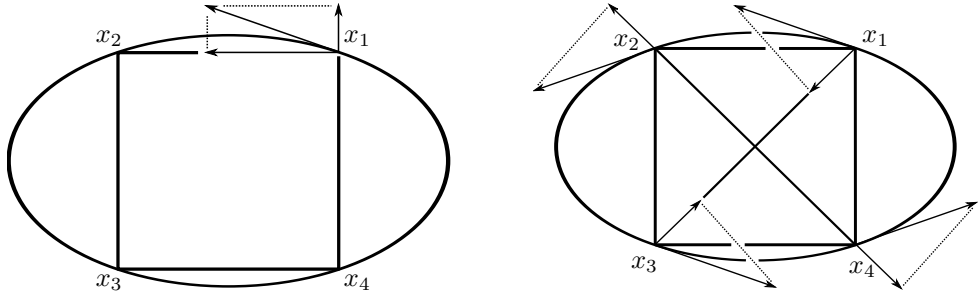
\begin{figure}[h]
\hfill
\begin{overpic}[height=1.5in]{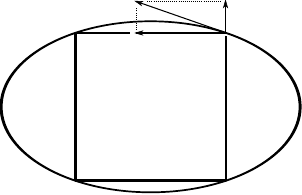}
\put(76,56){$x_1$}
\put(20,56){$x_2$}
\put(20,1){$x_3$}
\put(76,1){$x_4$}
\end{overpic}
\hfill
\begin{overpic}[height=1.5in]{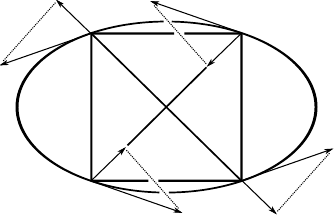}
\put(74,56){$x_1$}
\put(18,55){$x_2$}
\put(22,5){$x_3$}
\put(71,4){$x_4$}
\end{overpic}
\hfill
\hphantom{.}
\caption{These figures show that for a non-circular ellipse, $\cfog \transverse \slq$ at their intersections along the unique square inscribed in the ellipse. On the left, we see the effect of moving only $x_1 = (a \cos \theta_1,b \sin \theta_1)$ along the ellipse on the sides of the quadrilateral. This motion increases $\norm{x_4 - x_1}$ while decreasing $\norm{x_1 - x_2}$ and a calculation shows that it changes the edgelength ratio vector $(s_{142},s_{213},s_{324},s_{431})$ by a positive scalar multiple of $(a^2 + b^2, -a^2, 0, -b^2)$. On the right, we see the effect of moving all the $x_i$ simultaneously on the diagonals of the quadrilateral. The motion is decreasing the diagonal $\norm{x_1 - x_3}$ while increasing the diagonal $\norm{x_2 - x_4}$ and a calculation shows the difference of ratios $s_{132} - s_{241}$ decreases to first order.
\label{fig:ellipsetrans}}
\end{figure}

We will now write $\slq$ as the inverse image of $((1,1,1,1),0)$ under the map $f:\cfr \rightarrow \{xyzw = 1\} \subset \R^4 \cross \R$ given by $(s_{142},s_{213},s_{324},s_{431}) \cross (s_{132} - s_{241})$ and show that $\cfog$ is transverse to $\slq$ in $\cfr$ by showing that $f$ restricted to $\cfog$ is transverse to $((1,1,1,1),0)$.

Consider the effect of moving $x_1$ along the ellipse as shown on the left hand side of Figure~\ref{fig:ellipsetrans}. We saw above that this point is $(a \cos \theta_1,b \sin \theta_1)$ where $\cos \theta_1 = b/\sqrt{a^2 + b^2}$ and $\sin \theta_1 = a/\sqrt{a^2 + b^2}$, so the tangent vector to the ellipse $d\theta_1$ is $(-a^2/\sqrt{a^2 + b^2},b^2/\sqrt{a^2 + b^2})$. We can then compute the image of $d\theta_1$ under the differential of $(s_{142},s_{213},s_{324},s_{431})$ to be a positive scalar multiple (multiply by $2ab$) of 
\begin{equation*}
\vec{v}_1 = \left( a^2 + b^2, -a^2, 0, -b^2 \right)
\end{equation*}
Similarly, $d\theta_2$ and $d\theta_3$ are scalar multiples of cyclic permutations of $\vec{v}_1$. The Gram matrix of these vectors has determinant $4(a^6 + a^4 b^2 + a^2 b^4 + b^6)^2 \neq 0$. This shows that on $\cfog$, the differential $Df$ is onto the 3-dimensional tangent space to $\{xyzw = 1\} \subset \R^4$. 

We compute the image of $d\theta_1$ under the differential of $s_{132} - s_{241}$. If we use the facts that the sides and diagonals of the square are equal, this differential simplifies to a positive multiple of the derivative of the diagonal $\norm{x_1 - x_3}$, which can be written $\lambda^2 (b^2 - a^2)$. Tracking through what happens as we permute, we see that all the $d\theta_i$ are equal. Summing as in the right-hand side of Figure~\ref{fig:ellipsetrans} we see that this derivative does not vanish, so $s_{132} - s_{241} \transverse \{0\}$ at these points on $\cfog$. Together, we have proved that $f \transverse ((1,1,1,1),0)$ and hence that $\cfog \transverse \slq$. We conclude that the quotients $\qcfog$ and $\qslq$ are transverse as well.
\end{proof}

We can now complete the proof of Theorem~\ref{thm:squarepeg}. Given a $C^1$ curve in $\R^k$ we can find a nearby smooth curve $\gamma$. We claim that $\slq$ and $\cfog$ are boundary-disjoint in $\cfr$. Since $\cfog$ does not contact the $(13)(24)$ or $((13)(24))$ faces of $\cfr$, we need only consider the portion of $\slq$ on the interior of the $(1234)$ face. These configurations are infinitesimal tetrahedra with equal sides and equal diagonals. However, configurations on the $(1234)$ face of $\bdy \cfog$ are infinitesimally collinear configurations since $\gamma$ is smooth! This means that they have $\pi_{ij}$ and $s_{ijk}$ data very different from that of configurations in $\bdy\slq$.

We apply Theorem~\ref{thm:transversality-config} to perturb that smooth curve to a $C^1$-close curve $\eta$ with $C_4^0[\eta] \transverse \slq$. As transversality is a local property and the action of $\Z/4\Z$ is smooth, free, and properly discontinuous, this implies that $\hat{C}_4^0[\eta] \transverse \qslq$ as well.
As before, Haefliger's Theorem~\ref{thm:isotopy} guarantees a differentiable isotopy between the ellipse and $\eta$, and we can lift the isotopy to $\qcfr$, perturbing it without changing the ends so that it is transverse to $\qslq$ everywhere. This means that the finite collection of points (0-manifold) $\qcfog \cap \qslq$ is cobordant by a 1-manifold to the single square in the initial ellipse in $\qslq$, and hence that the number of inscribed squares is odd.
\end{proof}

A few historical comments are in order here. First, this is certainly not the first proof of the square-peg theorem to use an intersection-theoretic approach. Griffiths~\cite{MR1095236} took a similar approach, though he seems to have failed to appreciate the orientation-reversing nature of the cyclic permutation on $\cfr$. As a result, he (wrongly) computes a different intersection number to be 16 instead of zero, and claims as a result to have proved not only the square-peg theorem but a ``rectangular-peg theorem''. The rectangular case does not admit the quotient-space simplification above (there are generally \emph{two} inscribed rectangles of a given aspect ratio in the ellipse). As far as we know, the ``rectangular-peg theorem'' is an open and difficult problem. Matschke~\cite{2010arXiv1001.0186M} proved a version of the square-peg theorem from a theorem about loops of polygons inscribed in curves by arguing that a loop of rhombi which was invariant under the cyclic permutation contained a square by the intermediate value theorem, also an approach followed by 
Schnirel'man~\cite{vonSchnirelman:1944wf}. 

\subsection{Generic spheres have inscribed simplicies}

In this section, we explore a sort of reverse version of our basic framework. Previously, we used Haefliger's theorem to construct a map $E: S^l \cross I \rightarrow \R^k$ encoding the isotopy between our initial and target spheres that was transverse to $Z$ at both ends. But Haefliger's theorem really gives us a collection of diffeomorphisms $F_t$ of $\R^k$ parametrized by $t$ so that $F_0$ is the identity and $F_1$ maps our initial $S^l$ to the target $S^l$ and the compositions of $F_0$ and $F_1$ with our standard embedding were transverse to $Z$. Now we note that this construction works in reverse: Composing the inclusion $Z \hookrightarrow \cnr$ with the family $F_t^{-1}$ we get a family of maps $E:Z \cross I \rightarrow \cnr$ so that $E(-,0)$ and $E(-,1)$ are transverse to $C_n[S^l]$. Running through the rest of our standard argument, we see that $C_n[S^l] \cap Z$ and $C_n[\gamma'] \cap Z$ are cobordant in $Z \cross I$ and hence represent the same homology class in $Z$. 

For instance, if we let $Z = \text{equilateral triangles in $\R^2$}$, we could compute $H_1(Z;\Z) \simeq \Z$, because $Z$ deformation retracts to $S^1$, and then show that the submanifold of inscribed equilateral triangles in a curve represents $+1$ in $H_1(Z;\Z) = \Z$. We now prove a more general version of that theorem for inscribed simplices in higher-dimensional spheres, such as the inscribed regular tetrahedron in the irregular surface of Figure~\ref{fig:tetra}. To do so, recall that we showed in Proposition~\ref{prop:simpR is submanifold} that the space $\simpR$ of simplices in $\R^k$ with vertex-vertex distances in any nondegenerate, constructible ratio (cf. Definition~\ref{def:constructible}) is a submanifold of $C_{k+1}[\R^k]$ homotopic to $O(k)$. Also recall that $\simpR^+$ is the set of configurations in $\simpR$ where the matrix with columns $\pi_{1(k+1)}, \dots, \pi_{k(k+1)}$ has positive determinant.

\begin{figure}
\hfill
\includegraphics[height=1.65in]{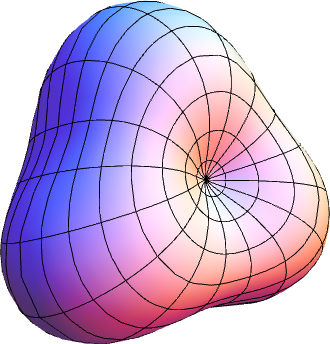}
\hfill
\includegraphics[height=1.65in]{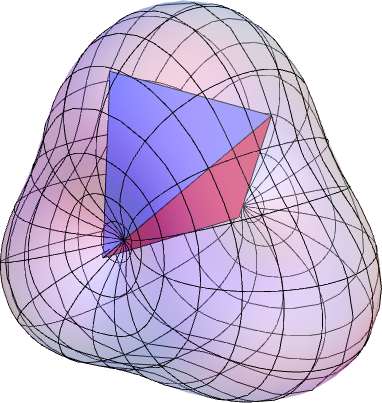}
\hfill
\includegraphics[height=1.65in]{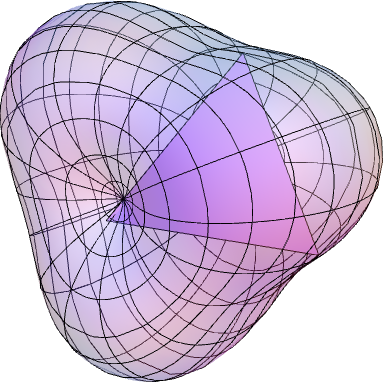}
\hfill
\hphantom{.}
\caption{On the left, we see an irregular embedding of $S^2$ in $\R^3$ described in spherical coordinates as a graph over the unit sphere by the function $r(\phi,\theta) = 1 + \sin^3\phi \sin 3\theta/5 - |\cos^7\phi|$. The center and right images show different views of a single regular tetrahedron inscribed in this surface with edgelengths close to $1.15$. If this embedding of $S^2$ is transverse to the submanifold of regular tetrahedra, this tetrahedron is a member of the 3-dimensional family of inscribed regular tetrahedra predicted by Theorem~\ref{thm:simplices}. This tetrahedron was found by computer search. Its vertices have spherical $(\phi,\theta)$ coordinates $(0.224399, 0.224399), (1.5708, 3.36599), (1.5708, 2.0196), (2.91719, 0.224399)$.
\label{fig:tetra}}
\end{figure}


\begin{theorem}[Inscribed Simplex Theorem]
For any $C^1$  embedding of $S^{k-1}$ in $\R^k$ and any nondegenerate, constructible simplex distance ratio $R$, there is a $C^1$-close embedding $\gamma$ so that $\simpR^+ \cap\, C_{k+1}[\gamma]$ is a 
smooth, orientable $k(k-1)/2$-dimensional manifold. Further, the projection $\pi:\simpR^+ \rightarrow SO(k)$ induces the map 
\begin{equation*}
\pi_* : H_{k(k-1)/2}(\simpR^+ \cap\, C_{k+1}[\gamma];\Z) \simeq \Z \rightarrow H_{k(k-1)/2}(SO(k);\Z) \simeq \Z, \qquad \pi_*(+1) = +1.
\end{equation*}
In particular, given a standard simplex $\Delta$ with distance ratio $R$ and any element $A \in SO(k)$, there is a scale and translation so that the scaled, translated copy of $A\Delta$ is inscribed in $\gamma$.
\label{thm:simplices}
\end{theorem} 

To get a sense of the meaning of this theorem, it tells us that any $C^1$ embedding of the sphere in $\R^3$ is $C^1$-close to an embedding with a 3-dimensional family of inscribed regular tetrahedra. Since the space of inscribed quadruples in a sphere is eight dimensional and the regularity of the tetrahedron is encoded by a specific ratio among six pairwise distances between vertices (a codimension five constraint), this result has at least the correct dimension (though it may be surprising that there is an entire $\SO(3)$ of inscribed tetrahedra in such a sphere!).

\begin{proof}
As before, we will follow our standard pattern: establish a base case and a modification of the given embedding that ensures a transverse intersection using boundary-disjointness of the two submanifolds of $C_{k+1}[\R^k]$, use Haefliger's theorem to find an isotopy, and use transversality to complete the proof. 

\begin{proposition}
If $S^{k-1}$ is the unit $(k-1)$-sphere, $\simpR^+ \transverse C_{k+1}[S^{k-1}]$ with $\pi\co (\simpR^+) \cap C_{k+1}[S^{k-1}] \rightarrow SO(k)$ a diffeomorphism.
\end{proposition}

\begin{proof}
We need another useful fact from distance geometry:
\begin{theorem}[Proposition 9.7.3.7 \cite{berger1987geometry}]
A simplex $\bfx = (\bfx_1, \dots, \bfx_{k+1}) \in C_{k+1}[\R^k]$ with pairwise distances $d_{ij}$ and $D(d_{11},\dots,d_{k+1,k+1}) > 0$ is inscribed in a unique $(k-1)$-sphere of radius $\rho(d_{11}, \dots, d_{k+1,k+1})$ where
\begin{equation}
\rho^2 = - \frac{
\left|
\begin{matrix}
0 & d_{12}^2 & \cdots & d_{1,k+1}^2 \\
d_{21}^2 & 0 & \cdots & d_{2,k+1}^2 \\
\vdots & \vdots & & \vdots \\
d_{k+1,1}^2 & d_{k+1,2}^2 & \cdots & 0 
\end{matrix}
\right|
}
{2 D(d_{11},\dots,d_{k+1,k+1})}
\label{eq:circumradius}
\end{equation}
\label{thm:circumradius}
\end{theorem}
Given any $\bfx \in \simpR^+ \cap\, C_{k+1}[S^{k-1}]$, the theorem immediately implies that the scale and position (of the circumcenter) of $\bfx$ are fixed, while the orientation of $\bfx$ is given uniquely by an element of $SO(k)$, proving the second half of the theorem. 

Proving transversality is more interesting. For $\bfx \in C_{k+1}[S^{k-1}]$, the orthogonal complement of $T_\bfx C_{k+1}[S^{k-1}]$ in $T_\bfx C_{k+1}[\R^{k}]$ is the $(k+1)$-dimensional space with orthonormal basis $\mathcal{B} = \{(\bfx_1, 0, \dots, 0), \dots, (0,\dots,\bfx_{k+1})\}$. The tangent space $T_\bfx \simpR^+$ contains the vectors $(\bfe_1, \dots, \bfe_1), \dots, (\bfe_k,\dots,\bfe_k)$ from the translational component of $\simpR^+$ as well as the vector $(\bfx_1,\dots,\bfx_{k+1})$ from scaling the configuration $\bfx$. Writing these vectors in the basis $\mathcal{B}$, we get the matrix:
\begin{equation*}
M = \left( 
\begin{matrix}
x_{1,1} & x_{2,1} & \cdots & x_{k+1,1} \\
x_{1,2} & x_{2,2} & \cdots & x_{k+1,2} \\
\vdots  & \vdots  &        &   \vdots  \\
x_{1,k} & x_{2,k} & \cdots & x_{k+1,k} \\
1       & 1       & \cdots & 1         \\
\end{matrix}
\right).
\end{equation*}
Subtracting the last column from the rest, we get 
\begin{equation*}
M' = \left( 
\begin{matrix}
\bfx_1 - \bfx_{k+1} & \bfx_2 - \bfx_{k+1} & \cdots & \bfx_k - \bfx_{k+1} & \bfx_{k+1} \\
0             & 0             & \cdots & 0             & 1       \\
\end{matrix}
\right).
\end{equation*}
The determinant of this matrix is $\pm 1$ multiplied by the determinant of the upper-left $k \cross k$ principal minor. But that determinant is positive because $\bfx \in \simpR^+$. 
\end{proof}


\begin{proposition}
If $\gamma$ is a smooth embedding of $S^{k-1}$ in $\R^k$ and $R$ is a constructible and nondegenerate simplex distance ratio, the smooth submanifolds $C_{k+1}[\gamma]$ and $\simpR$ of $C_{k+1}[\R^k]$ are boundary disjoint.
\label{prop:simpR and cgamma are boundary disjoint}
\end{proposition}

\begin{proof}
Since $R$ is nondegenerate, $\bdy \simpR$ is contained in the $(1 \cdots k+1)$ face of $\bdy C_{k+1}[\R^k]$ where all points come together. 

The collection of $\pi_{ij}$ maps determines a continuous map $\Pi : C_{k+1}[\R^k] \rightarrow (S^{k-1})^{k(k+1)}$. Further, $SO(k)$ acts diagonally on both sides of this map. Since $R$ is nondegenerate, for any $\bfx \in \simpR$, the directions in $\Pi(\bfx)$ do not lie on any great $S^{k-2}$ (otherwise, the simplex would lie in a hyperplane and hence have zero volume). Let $\chi(p)$ be the squared distance between a point configuration in $(S^{k-1})^{k(k+1)}$ and the nearest configuration in a (diagonal) great $(S^{k-2})^{k(k+1)}$. Since $\chi(p)$ is invariant under the diagonal action of $SO(k)$ on $(S^{k-1})^{k(k+1)}$, and $\Pi(\bfx)$ is invariant under translation and scaling, the map $\chi \circ \Pi$ is constant and nonzero on $\simpR$. However, the infinitesimal configurations in the $(1 \cdots k+1)$ face of $\bdy C_{k+1}[\gamma]$ do lie in a great $S^{k-2}$ determined by the tangent space to $\gamma$, and so $\chi = 0$ on this face of $\bdy C_{k+1}[\gamma]$. This implies that $\bdy \simpR$ and $\bdy C_{k+1}[\gamma]$ are disjoint, as desired.
\end{proof}

Given a $C^1$  embedding of $S^{k-1}$ in $\R^k$, we can smooth it and apply Theorem~\ref{thm:transversality-config} to find a $C^1$-close smooth $(k-1)$-sphere $\gamma$ with $C_{k+1}[\gamma] \transverse \simpR^+$, using Proposition~\ref{prop:simpR and cgamma are boundary disjoint} to show that $C_{k+1}[\gamma]$ and $\simpR^+$ are boundary disjoint, as required by Theorem~\ref{thm:transversality-config}.

As before, Haefliger's Theorem~\ref{thm:isotopy} guarantees a differentiable isotopy between the standard unit $S^{k-1}$ and $\gamma$. The new step is that we invert this isotopy to get a map $E\co\simpR^+ \cross I \rightarrow C_{k+1}[\R^k]$, so that $E(-,0)$ and $E(-,1)$ are transverse to the standard unit $S^{k-1}$, $E(\simpR^+,0)$ is the standard $\simpR^+$, and there's a diffeomorphism of $\R^k$ which carries $E(-,1)$ to $\simpR^+$ and the standard $S^{k-1}$ to $\gamma$. The rest of the proof goes as before. 
\end{proof}


\section{Future Directions}

One of the recurring features of this work is that the introduction of compactified configuration spaces simplifies many of the tricky technical pieces in the proof by exporting the troublesome behavior to the boundaries. For example, applying a transversality theorem to squares and configurations of inscribed quadrilaterals requires us to have some strategy for dealing with ``degenerate'' configurations. The extension of the $\pi_{ij}$ and $s_{ijk}$ data to the boundary of configuration space (with the associated metric) allowed us to argue easily that there could be no infinitesimal squares inscribed on a smooth curve. On the other hand, this is not the only way to address these difficulties: For instance, Stromquist~\cite{MR1045781} deals with basically the same problem by showing directly that there are no squares (or square-like quadrilaterals) smaller than some $\epsilon$ which can be inscribed on a curve with some mild smoothness assumptions and hence avoids the dangerous diagonals of the product space $(\R^k)^4$. We give a similar argument in the Appendix to show:
\begin{theorem}
Any closed curve in $\R^n$ of finite total curvature with no cusps has at least one inscribed \sqrtextp
\label{prop:ftcwc}
\end{theorem}
We note that since this result is obtained by a limit argument, we cannot rule out the possibility that several squares come together in the limit to leave an even number of squares inscribed in the final curve, as in the examples of \cite{2008arXiv0810.4806P}. The appeal of this result is largely that the class of curves of finite total curvature is a well-understood space~(cf.~\cite{math.GT/0606007}). It is not hard to see that Stromquist's theorem~\cite{MR1045781} is more general.

A very interesting possible extension of the methods here would be to use the 1-jet version of multijet transversality to try to prove a transversality theorem for submanifolds of configuration spaces which do intersect in certain boundary faces. Doing so would allow one to extend the ``counting'' and homology arguments above to detect boundary intersections between submanifolds of configuration spaces. For example, one might try to argue in this way that the space of triangles with a given angle inscribed in a curve had the homology of the torus, keeping in mind that a circle's worth of such ``triangles'' would be expected to be chords meeting the tangent to the curve in the specified angle.
Another interesting use for such a theorem would be to try to extend these theorems to immersed curves with normal crossings (as opposed to simply studying embedded curves). 

We have proved that the space of curves with an odd number of squares are $C^1$-dense among $C^1$ curves in the plane (or residual among smooth curves). This is not quite the same as proving that a ``generic'' $C^1$ curve has an odd number of inscribed squares. It would be very interesting to try to extend these results to a set of curves which was full-measure among plane curves according to some natural measure on curves, as Morgan does in~\cite{1978InMat..45..253M} for space curves bounding a unique area-minimizing surface.

\begin{acknowledgements}
The authors would like to first thank Gerry Dunn who introduced us to the problem. We would also like to thank the people who have discussed the problem with us over the years: Jordan Ellenberg, Richard Jerrard, Rob Kusner, Benjamin Matschke,  Igor Pak, Strashimir Popvassiliev, John M. Sullivan, Cliff Taubes, and Gunter Ziegler.
\end{acknowledgements}

\bibliography{ngons,oldreferences,paperslib}{}
\bibliographystyle{plain}

\newpage
\appendix*
\section{Finite Total Curvature Curves without Cusps}

We have shown that every $C^1$ curve in $\R^k$ is $C^1$-close to a smooth curve with an odd number of inscribed \sqrtextsp This means that any curve which may be approximated by a sequence of $C^1$ curves may be approximated by a sequence of smooth curves with inscribed squares. Can we use this argument to extract at least one limiting inscribed \sqrtext on any curve in $\R^k$? The problem is clear: The sequence of \sqrtexts on the approximating curves may have sidelengths approaching zero. If one could construct a general lower bound on these sidelengths in terms of the global geometry of the ``host'' curve, this possibility could be ruled out. We do not know of any explicit example of a family of curves where all the inscribed \sqrtexts have sidelengths converging to zero, so this approach may yet be possible. However, this line of attack has been more or less obvious from the start, and nobody has managed to construct such an argument in the past century. 

Our considerably more modest goal in this section is to rule out small \sqrtexts using \emph{local}, rather than global, data about the limit curve, and in this way to extend our results to the class of curves of \emph{finite total curvature without cusps} ($\FTCWC$), which is defined below.

Our argument has three parts. First, we show that each curve $\gamma$ in $\FTCWC$ has no inscribed \sqrtexts with side length smaller than a positive constant, denoted by $\operatorname{\pi-d}(\gamma)$. Next, we show that $\gamma$ is the limit of a sequence of smooth curves $\gamma_i$, for which $\operatorname{\pi-d}(\gamma_i) \rightarrow \operatorname{\pi-d}(\gamma)$, each containing an odd number of inscribed \sqrtextsp  The first two steps then imply that this sequence of \sqrtexts has a convergent subsequence with limit a \sqrtext with sidelength at least $\operatorname{\pi-d}(\gamma)$.

We recall some standard facts about curves of finite total curvature~\cite{math.GT/0606007, MR0176402}. The \emph{total curvature} of a curve is the supremum of the total turning angles of all polygons inscribed in the curve. If this supremum is finite, we say the curve has \emph{finite total curvature} or is in $\FTC$. 
Curves in $\FTC$ have a number of desirable properties. They are always rectifiable, and so can be parametrized by arclength. They are almost everywhere differentiable, and a curve in $\FTC$ has one-sided tangent vectors at every point. In fact, these tangents differ only at countably many corner points. There is a Radon measure $\kappa$ on every $\gamma$ in $\FTC$ whose mass on any open subarc of $\gamma$ is the total curvature (in the above sense) of the subarc. This measure has a countable number of atoms at corners of the curve $\gamma$. The mass of each atom is the turning angle between these vectors. If this turning angle is $\pi$, we say the corner is a \emph{cusp}. 

Since $\FTC$ curves have a second derivative (at least weakly) it is natural to want to approximate them ``in $C^2$'' by smooth curves. Unfortunately, this is not quite possible. Note that the tangent indicatrix to an $\FTC$ curve has gaps at the corners of the curve, while the tangent indicatrix of a smooth curve forms a continuous curve on $S^2$. Thus the tangent vectors to a sequence of smooth curves approximating an $\FTC$ curve can't converge to tangents of the $\FTC$ curve near a corner of the $\FTC$ curve. However, we can come very close to a $C^2$ approximation in the following sense:

\newpage

\begin{definition}
Suppose $\gamma$ is an $\FTC$ curve. Let $\Len(\gamma,a,b)$ be the length of the arc of $\gamma$ between $\gamma(a)$ and $\gamma(b)$ and $\kappa(\gamma,a,b)$ be the total curvature of this arc. We say that a sequence of finite total curvature curves $\gamma_i$ approximate $\gamma$ \emph{uniformly in position, arclength, and total curvature} if there are parametrizations of the $\gamma_i$ so that for each $\epsilon > 0$ there exists an $N$ so that for all $i > N$, we have the following:
\begin{enumerate}
\item For any $a$, $\norm{\gamma_i(a) - \gamma(a)} < \epsilon$.
\item For any arc $(a,b)$, $\norm{\Len(\gamma_i,a,b) - \Len(\gamma,a,b)} < \epsilon$.
\item For any arc $(a,b)$, $\norm{\kappa(\gamma_i,a,b) - \kappa(\gamma,a,b)} < \epsilon$.
\end{enumerate}
\label{def:uniform in pos len tc}
\end{definition}

\begin{proposition} 
\label{prop:approx}
Any $\FTC$ curve $\gamma$ may be approximated uniformly in position, arclength, and total curvature by smooth $\FTC$ curves $\gamma_i$.
\end{proposition}

\begin{proof}
This is an assembly of standard results about $\FTC$ curves. If we inscribe polygons with vertices equally spaced by arclength in $\gamma$, and parametrize them compatibly (so that the vertices have the same parameter values on $\gamma$ and on each polygon), the polygons converge uniformly in position and total curvature (cf. Lemma~4.2 of \cite{MR3048513}) and are all finite-total curvature curves (since their total curvatures are bounded by that of $\gamma$). 

To see that they converge uniformly in arclength, fix an arc $(a,b)$ of $\gamma$, and observe that the corresponding arcs of the $\gamma_i$ have bounded total curvature, and converge to the arc of $\gamma$ in Fr\'echet distance because they converge in position. Then use Theorem~5.1 of \cite{MR0176402} (see also \cite{MR2352603}) which states that for any rectifiable curves $K$, $L$,
\begin{equation*}
|\Len(K) - \Len(L)| \leq \delta(K,L) (\pi \max{\TC(K),\TC(L)} + 2)
\end{equation*}
where $\delta(K,L)$ is the Fr\'echet distance between $K$ and $L$. Note that this theorem is not obvious: it says that the standard examples of curves which converge in Fr\'echet distance but \emph{not} in arclength, such as a stairstep curve converging to the diagonal of a square, must all have unbounded total curvature. 

To finish the proof, smooth each polygon by rounding off corners-- the smooth curves have the same total curvature as the polygons (and are hence $\FTC$) and are close to the original polygons in position, arclength, and total curvature, as required.
\end{proof}


Notice that if a \sqrtext $pqrs$ is inscribed in an arc of $\gamma$, the total curvature of the arc $\arc{pqrs}$ must be at least as large as the total curvature (or total turning angle) of the inscribed polygon $pqrs$. If $pqrs$ is a planar square, it is clear that this turning angle is $\pi$. We now prove that the turning angle is at least $\pi$ if $pqrs$ is a \sqrtextp

\begin{lemma} \label{lem:sqrturn} Any \sqrtext $pqrs$ has the property that $\kappa(pqrs) \geq \pi$, with equality if and only if $pqrs$ is a planar square.
\end{lemma}

\begin{proof}
Consider the situation of Figure~\ref{fig:tetraturning} where $pqrs$ has equal sides $pq$, $qr$, $rs$, and $sq$ and equal diagonals $pr$ and $qs$.
\begin{figure}
\begin{overpic}{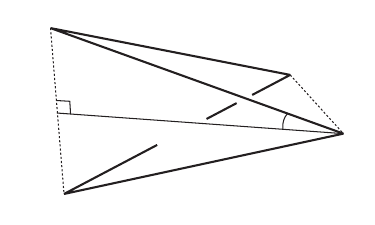}
\put(92.5,22){$p$}
\put(10,52){$q$}
\put(13.8,6){$s$}
\put(78,41){$r$}
\put(42,22.7){$\cos \theta$}
\put(3,39){$\sin\theta$}
\put(12,29){$t$}
\put(70,26.5){$\theta$}
\end{overpic}
\caption{The arc $\arc{pqrs}$ of the \sqrtext shown has total curvature given by $2\pi - 4\theta$. We observe, however, that $pt$ has length $\cos \theta$ and $qt$ has length $\sin \theta$, while $2qt = qs = pr$ is less than $2pt$. Thus $\cos \theta \geq \sin \theta$ and $\theta \leq \pi/4$.}
\label{fig:tetraturning}
\end{figure}
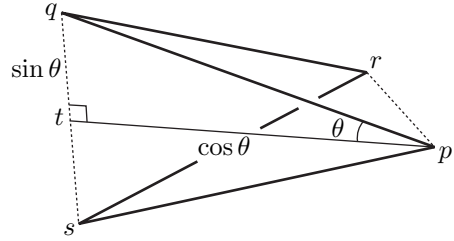 
We may assume without loss of generality that the sides have length 1. We construct the midpoint $t$ of $qs$. Since $\triangle pqs$ is isosceles, we can conclude that $\angle qps = 2 \theta$ and that $\angle ptq$ is right. We then have $pt = \cos \theta$ and $qt = \sin \theta$. Further, since $qs = pr$, we have $pr = 2 \sin \theta$.

Since $pq = rq$ and $ps = rs$, we have $\triangle rqs \cong \triangle pqs$. Thus $\angle qrs = \angle qps = 2 \theta$ and as above $rt = \cos \theta$. So by the triangle inequality (on $\triangle ptr$) we have $pt + tr \geq pr$, or
\begin{equation*}
2 \cos \theta \geq 2 \sin \theta.
\end{equation*}
This means that $\theta \leq \pi/4$, and $\theta = \pi/4$ if and only if $t$ is on the line $pr$. In this case $pqrs$ is planar (and hence it is a square). Now the turning angle of the arc $pqrs$ is $\pi - 2\theta$ at $q$ and $r$. Thus the total turning angle of $pqrs$ is $2 \pi - 4\theta \geq \pi$, as desired.
\end{proof}


Our overall goal is to prove that there exists an $\epsilon > 0$ for each curve in $\FTCWC$ so that no \sqrtext inscribed in $\gamma$ has sidelength less than $\epsilon$.


\begin{definition}
\label{def:par}
We define the \textbf{$\pi$-distance} of an FTC curve $\gamma$, denoted $\operatorname{\pi-d}(\gamma)$.
The value $\ell$ is an {\it admissible distance bound} if every open subarc $(a,b)$ of $\gamma$ with $\norm{\gamma(a) - \gamma(b)} < \ell$ has $\kappa(\gamma,a,b) < \pi$. Then 
\begin{equation*}
\operatorname{\pi-d}(\gamma) =  \sup_{\ell \text{ is admissible}} \ell  = 
\inf_{\ell \text{ is inadmissible}} \ell.
\end{equation*}
\end{definition}

Note that if $\ell$ is inadmissible, then there is some subarc $(a,b)$ with $\norm{\gamma(a) - \gamma(b)} < \ell$, but $\kappa(\gamma,a,b) \geq \pi$. The point of $\operatorname{\pi-d}(\gamma)$ is that it provides a lower bound on the side length of a \sqrtext inscribed in $\gamma$.

\begin{lemma}\label{lem:pi d}
Any \sqrtext inscribed in an FTC curve $\gamma$ has sidelength greater than or equal to $\operatorname{\pi-d}(\gamma)$. 
\end{lemma}

\begin{proof}
Let $pqrs$ be an inscribed \sqrtext in $\gamma$, and consider the arc $pqrs$ which has end-to-end distance $\norm{\gamma(p) - \gamma(s)}$. By Lemma~\ref{lem:sqrturn}, the \sqrtext is an inscribed polygon with total curvature at least $\pi$. Thus $\kappa(\gamma,p,s) \geq \pi$. This means that $\norm{\gamma(p) - \gamma(s)}$ is an inadmissible distance bound, and hence it is at least $\operatorname{\pi-d}(\gamma)$, as desired.
\end{proof}

We now want to show that an embedded curve in $\FTCWC$ is the limit of a sequence of smooth curves with inscribed \sqrtexts with side lengths uniformly bounded above zero. We proceed in two steps: first we'll show that $\gamma$ itself has $\operatorname{\pi-d}$ bounded above, then that $\operatorname{\pi-d}$ behaves nicely under the sort of convergence of curves we introduced above.

\begin{lemma}
\label{lem:padgzero}
If $\gamma$ is an embedded curve in $\FTCWC$, then $\operatorname{\pi-d}(\gamma) > 0$.
\end{lemma}

\begin{proof} 
Suppose not. Since $\operatorname{\pi-d}(\gamma) = 0$, there is a sequence of inadmissible $\ell_i \rightarrow 0$. So there exists a collection of open subarcs $A_i$ of $\gamma$ whose endpoints $a_i$, $b_i$ have $\norm{\gamma(a_i) - \gamma(b_i)} \rightarrow 0$, while $\kappa(\gamma,a_i,b_i) \geq \pi$. Passing to a subsequence where $a_i \rightarrow a$ and $b_i \rightarrow b$, we see that $\gamma(a) = \gamma(b)$, and hence $a = b$ because $\gamma$ is embedded. 

Now as the $A_i$ approach $\{a\}$, their total curvature $\kappa(A_i) \geq \pi$. Since $\gamma$ is compact, we may pass to a subsequence of $A_i$ that are nested and converge to a point $p$. Since $\kappa$ is an outer-regular measure, this means that $\kappa(p) \geq \pi$. Since $\kappa(p)$ is a turning angle, it is always $\leq \pi$. Thus $\kappa(p) = \pi$ and $p$ is a cusp point, contradicting our assumption that $\gamma$ was in $\FTCWC$.
\end{proof}

Since $\operatorname{\pi-d}$ is defined by lengths, distances, and curvatures, we can expect it to behave nicely as we take limits in the sense of Definition~\ref{def:uniform in pos len tc}.

\begin{proposition}
If $\gamma_i \rightarrow \gamma$ uniformly in position, arclength, and total curvature in the sense of Definition~\ref{def:uniform in pos len tc}, and $\operatorname{\pi-d}(\gamma) > 0$, then $ \varliminf_{i\rightarrow \infty} \operatorname{\pi-d}(\gamma_i) > 0$.
\label{prop:bounded pi d}
\end{proposition}

\begin{proof}
Suppose not. For any $\epsilon > 0$, there must be infinitely many $\gamma_i$ with $\operatorname{\pi-d}(\gamma_i) < \epsilon$. Each $\gamma_i$ contains a subarc $(a_i,b_i)$ with $\norm{\gamma_i(a_i) - \gamma_i(b_i)} < \epsilon$, but $\kappa(\gamma_i,a_i,b_i) \geq \pi$. By compactness, we can assume that we have passed to a  subsequence where $a_i \rightarrow a$ and $b_i \rightarrow b$. 

Now by convergence in position, $\norm{\gamma(a) - \gamma(b)} \leq \epsilon$. Let us expand the open arc $(a,b)$ of $\gamma$ slightly to an open subarc $(a',b')$ with $\norm{\gamma(a') - \gamma(b')} \leq 3 \epsilon$, say, and again pass to a subsequence where $(a_i,b_i) \subset (a',b')$ for all $i$. Now for any $\delta > 0$, by convergence in total curvature, for large enough $i$ we have
\begin{equation*}
\norm{\kappa(\gamma,a',b') - \kappa(\gamma_i,a',b')} < \delta
\end{equation*}
so
\begin{equation*}
\kappa(\gamma,a',b') > \kappa(\gamma_i,a',b') - \delta \geq \kappa(\gamma_i,a_i,b_i) - \delta \geq \pi - \delta.
\end{equation*}
where $\kappa(\gamma_i,a',b') \geq \kappa(\gamma_i,a_i,b_i)$ because $(a_i,b_i) \subset (a',b')$. Since $\delta$ was arbitrary, this proves that $\kappa(\gamma,a',b') \geq \pi$. 

However, this means that $3 \epsilon > \norm{\gamma(a') - \gamma(b')}$ is an inadmissible distance bound for $\gamma$, and hence that $\operatorname{\pi-d}(\gamma) < 3 \epsilon$. Since $\epsilon$ was arbitrary, this proves that $\operatorname{\pi-d}(\gamma) = 0$, providing the required contradiction.
\end{proof}

We are ready to construct an inscribed \sqrtext on any $\FTCWC$ curve. We have done all the hard work above; it remains only to assemble the component pieces.

\begin{theorem}
\label{thm:FTCWC}
There is an inscribed \sqrtext on any embedded curve $\gamma$ in $\FTCWC$. In particular, there is an inscribed \sqrtext on any embedded $C^2$-smooth curve $\gamma$.
\end{theorem}

\vspace{-0.25in}

\begin{proof}
First, we may approximate $\gamma$ by a sequence of smooth curves $\gamma_i$ with convergence in position, arclength, and total curvature by Proposition~\ref{prop:approx}. By making a $C^2$-small perturbation of each $\gamma_i$, we may assume by Theorem~\ref{thm:squarepeg} that each $\gamma_i$ contains at least one inscribed \sqrtext. Since our perturbations were $C^2$-small, the sequence of curves $\gamma_i$ still enjoys finite total curvature and converges to $\gamma$ in position, arclength, and total curvature.

By Lemma~\ref{lem:padgzero} and Proposition~\ref{prop:bounded pi d}, there is an $\epsilon>0$ so that we may pass to a subsequence of $\gamma_i$, each of which has $\operatorname{\pi-d}(\gamma_i) > \epsilon$. By Lemma~\ref{lem:pi d} the inscribed \sqrtext on each $\gamma_i$ has sidelength at least $\epsilon$. This is the crucial point in the proof: by bounding the sidelengths of these \sqrtexts below, we have ensured that they do not shrink away as we approach the limiting curve $\gamma$. 

From here, the argument is standard. We may assume that the inscribed \sqrtexts in the $\gamma_i$ lie in a compact subset of $\slq$, and hence that they have a convergent subsequence. The limit of this subsequence is a \sqrtext inscribed in the limit curve $\gamma$. 
\end{proof}

Note that we have lost something here: it is possible that multiple \sqrtexts coincide on the limiting curve $\gamma$, so the count of inscribed \sqrtexts may no longer be odd, as shown by the examples of Popvassiliev~\cite{2008arXiv0810.4806P}. 

Also note that there exist $C^1$ curves that are not FTC; these don't have corners but have spirals where curvature diverges. For these curves, Theorem~\ref{thm:squarepeg} still holds, but we can not conclude from Theorem~\ref{thm:FTCWC} that there is at least one \sqrtext. (The spirals prevent the arguments of Proposition~\ref{prop:bounded pi d} from holding.)
\end{document}